\documentclass[12pt, reqno]{amsart}

\usepackage{amssymb,latexsym,amsthm, amsmath, comment, fullpage}
\usepackage[all]{xy}
\usepackage{graphicx, comment, libertine, wrapfig}

\ifx\pdfoutput\undefined
\usepackage{hyperref}
\else
\usepackage[pdftex,colorlinks=true,linkcolor=blue,urlcolor=blue]{hyperref}
\fi

\theoremstyle{plain}
\newtheorem{theorem}{Theorem}

\newtheorem{corollary}{Corollary}
\newtheorem{proposition}{Proposition}

\newtheorem{lemma}{Lemma}
\theoremstyle{definition}
\newtheorem{definition}{Definition}
\newtheorem{remark}{Remark}

\newtheorem*{ack}{Acknowledgements}

\def\tr{\mathrm{Tr}}

\begin{document}

\title[Tilings, traces and triangles]
      {Tilings and traces}
\author{Rodrigo Trevi\~no }
\address{University of Maryland, College Park}
\email{rodrigo@umd.edu}
\date{  \today}
      \begin{abstract}
This paper deals with (globally) random substitutions on a finite set of prototiles. Using renormalization tools applied to objects from operator algebras we establish upper and lower bounds on the rate of deviations of ergodic averages for the uniquely ergodic $\mathbb{R}^d$ action on the tiling spaces obtained from such tilings. We apply the results to obtain statements about the convergence rates for integrated density of states for random Schr\"odinger operators obtained from aperiodic tilings in the construction.
      \end{abstract}
      \maketitle
      \section{Introduction}
Consider the two substitution and expansion rules defined on the half hexagons in Figure \ref{fig:HalfHexes}, one of which is the classical half hexagon substitution rule and the second one is obtained by modifying the square (second iteration) of it.      
This paper is concerned about the random application of substitution and expansion rules such as these in order to construct aperiodic tilings of $\mathbb{R}^d$, the study of the statistical properties of such tilings, and an application to the study of random Schr\"odinger operators on quasicrystals. Figure \ref{fig:randomPenrose} gives an example of the types of tilings one can get through random application of the substitution rules in Figure \ref{fig:HalfHexes}.
      \begin{figure}[h]
        \centering
        \includegraphics[width = 4.5in]{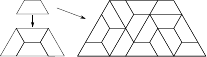}
        \caption{ Two substitution and expansion rules on half hexagons.}
        \label{fig:HalfHexes}
      \end{figure}

      Although first introduced in \cite{GL:random}, interest in random substitution tilings has surged recently, (e.g. \cite{FrankSadun:fusion, GM:mixed, BD:sadic, Rust:uncountable, RS:random, ST:random}). Random substitutions come in two flavors: locally random constructions (e.g. \cite{RS:random}) and globally random constructions (e.g. \cite{GM:mixed, ST:random}). The typical features of globally random tilings are repetitivity, uniform patch frequencies (equivalent to unique ergodicity, see \S\ref{sec:tilings}) and zero entropy whereas locally random tilings typically have positive entropy and non-uniform patch frequencies. This distinction is similar to that between strictly ergodic and mixing subshifts. All of the constructions in this paper are of the globally random flavor and so, although it will not be stated repeatedly that they are of global type, the reader should assume so throughout paper.

      The present work can be seen as an extension or alternative to the work \cite{ST:random}. It is an extension because the class of functions for which theorems are proved here (Lipschitz functions) is much larger than the class of functions treated in that work (smooth transversally locally constant functions). It is also an alternative because the present paper develops new tools combining ideas of renormalization with objects from operator algebras. More specifically, an object called the \textbf{trace cocycle} is introduced, developed, and used here to obtain results on deviations of ergodic integrals for Lipschitz functions on tiling spaces coming from random substitutions. This new tool makes it possible to connect some of the invariants from AF algebras (the traces) with invariants (also traces) from certain ``smooth'' sub-algebras of the so-called algebras of \emph{random Schr\"odinger operators} on aperiodic tilings while giving errors of convergence rates for the Shubin-Bellissard formula.

      What both of these approaches have in common is the use spaces of Bratteli diagrams to organize tilings which can be constructed from applications of substitution rules defined on the same set of prototiles, and the use of subshifts as a ``moduli space'' of all tilings which can be obtained from a finite set of substitution rules, whereon the shift dynamics become renormalization dynamics. In \cite{ST:random} the topology of the resulting tiling spaces was well-studied and exploited to obtain statistical results for the tilings.

      In this paper, the renormalization approach is applied to certain invariants (the traces) from operator algebras to study the properties of the random substitution tilings, although they are close in spirit to the tools used by Bufetov in his study of deviation of ergodic integrals for several classes of systems \cite{bufetov:limit, bufetov:limitVershik, bufetov-solomyak:tilings}. What is gained from this point of view is that there is no need to have a full understanding of the topology of the tiling spaces constructed at random, making computations easier to make, as section \ref{sec:experiments} here demonstrates; what is lost is the access to topological information of the tiling spaces constructed in the construction.
      

     
\begin{figure}[t]
 \centering
 \includegraphics[width = 6.5in]{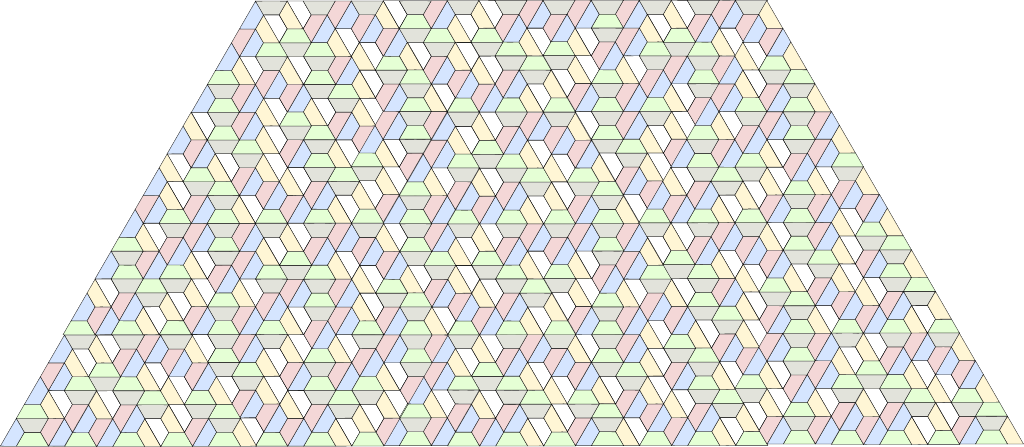}
  \caption{{\tiny A patch obtained from random applications of the half hex substitution and expansion rules in Figure \ref{fig:HalfHexes}.}}
  \label{fig:randomPenrose}
\end{figure}
Here progress is also made with the issue of boundary effects. By ``boundary effects'' I mean the following: in most studies of uniquely ergodic  $\mathbb{R}^d$-actions on metric spaces, when $d>1$, it has been usually hard to obtain information of the error terms of ergodic integrals of functions over sets of volume $\sim T^d$ which are smaller than $T^{d-1}$, which is the contribution of the boundary of the averaging set to the integral (\cite{sadun:exact, bufetov-solomyak:tilings, ST:SA, ST:random}).
These issues have been overcome in other settings of higher rank abelian actions (e.g. \cite{CosentinoFlaminio}), but they have remained an obstacle in the study of tilings. In this paper I show that given some set $B$ there is a set arbitrarily-close set $B_\varepsilon$ and a set of dilations of $B_\varepsilon$ such that the deviation behavior along those averaging sets are fully described by the Lyapunov spectrum of our renormalization cocycle. As the title suggests and it was suggested above, functionals from operator algebras called \textbf{traces} play a prominent role here, being the analogue to cycles in Zorich's theory \cite{zorich-leaves}, currents in Forni's theory \cite{forni:deviation}, and finitely-additive measures in Bufetov's theory\footnote{Ian Putnam recently pointed out to me that \cite[Theorem 2.1]{BowenFranks:homology} shows that the space of traces considered here and the space of finitely additive measures which Bufetov considers are isomorphic.} \cite{bufetov:limit}. Our cocycle is defined on a bundle of traces analogous to the cohomology bundle used for the Kontsevich-Zorich cocycle.

  
      Given that aperiodic tilings serve as models for quasicrystals, the results on deviations of ergodic averages here have several applications in mathematical physics. The advantage here of using an operator algebra approach is that it makes the connection to the study of random Schr\"odinger operators more natural.   In \cite{ST:traces} it was shown that asymptotic properties of traces of random Schr\"odinger operators defined by certain self-affine aperiodic tilings are controlled by traces obtained through the behavior of ergodic integrals on the tiling space. Here a generalization is made and the connection is made more explicit: since traces on locally finite subalgebras of AF algebras control the behavior of the ergodic integrals for randomly constructed tilings, one can obtain traces on algebras of operators which control the asymptotic properties of the integrated density of states for so-called random Schr\"odinger operators.
\subsection{Statement of results}
Let $\Sigma_N$ be the full $N$-shift, that is, the space of bi-infinite sequences of symbols from an alphabet of $N$ symbols. Given a set of prototiles $\{t_1,\dots, t_N\}$ and $N$ uniformly expanding and compatible substitution rules $\mathcal{F} = \{\mathcal{F}_1,\dots, \mathcal{F}_N\}$ on them (see the precise definition of substitution rule in \S \ref{sec:tilings}), there is a subshift of finite type $X_\mathcal{F}\subset \Sigma_N$ which parametrizes all the tiling spaces which can be obtained by random applications of the substitution rules in $\mathcal{F}$: given $x\in X_\mathcal{F}$ there is a corresponding compact metric space (called a \textbf{tiling space}) $\Omega_x$ whose elements are tilings with heirarchical structure dictated by the point $x$ according to the substitution rules in $\mathcal{F}$. Periodic points in $X_\mathcal{F}$ give rise to tiling spaces $\Omega_x$ consisting of self-similar tilings.

The tiling spaces admit a $\mathbb{R}^d$ action which is denoted by $\varphi_t:\Omega_x\rightarrow \Omega_x$ and for many of them this action is minimal and uniquely ergodic (this will be the scenario considered in this paper; see Proposition \ref{prop:UE}). The concept of a \textbf{minimal measure} is used here (see \S \ref{sec:blow} for the precise definition), and this roughly means that $\mu$ on $X_\mathcal{F}\subset \Sigma_N$ is minimal if for $\mu$-almost every $x\in X_\mathcal{F}$, $\Omega_x$ admits a minimal $\mathbb{R}^d$ action. The shift map $\sigma:X_\mathcal{F}\rightarrow X_\mathcal{F}$ defines a homeomorphism $\Phi_x:\Omega_x\rightarrow \Omega_{\sigma(x)}$ of tiling spaces. As such, the shift drives the renormalization dynamics.

The way of constructing $\Omega_x$ from $x$ is through a Bratteli diagram $\mathcal{B}_x$: a point $x\in X_\mathcal{F}$ establishes how a sequence of substitutions from the family $\mathcal{F}$ are put together to obtain a tiling, and this sequence is represented by an infinite directed graph $\mathcal{B}_x$ whose structure is tied to that of $\Omega_x$. As such, any point $x\in X_\mathcal{F}$ defines a $*$-algebra $LF(\mathcal{B}_x^+)$, called a \textbf{locally finite algebra} (this is defined in \S \ref{sec:LF}), which is dense in an approximately finite dimensional (AF) $C^*$-algebra $AF(\mathcal{B}^+_x)$. The dual of $LF(\mathcal{B}_x^+)\subset AF(\mathcal{B}^+_x)$ is the \textbf{trace space} of $LF(\mathcal{B}_x^+)$, which is a finite dimensional vector space over $\mathbb{C}$. Here, a trace $\tau$ on a $*$-algebra $\mathcal{A}$ is taken to be any linear map $\tau:\mathcal{A}\rightarrow \mathbb{C}$ satisfying $\tau(ab) = \tau(ba)$. Note that since every element of $K_0(AF(\mathcal{B}^+_x))$ can be represented by an element in $LF(\mathcal{B}^+_x)$, the space of traces $\tr(\mathcal{B}^+_x)$ can be seen as the dual to $K_0(AF(\mathcal{B}^+_x))$. The dual to $\tr(\mathcal{B}^+_x)$ as a vector space is the space of cotraces $\tr^*(\mathcal{B}^+_x)$ and it is this space which has great importance. We define the \textbf{trace bundle} to be the set of pairs $(x,\tau')$ with $\tau'\in \tr^*(\mathcal{B}^+_x)$. The shift $\sigma:X_\mathcal{F}\rightarrow X_\mathcal{F}$ induces a linear map $\sigma_*:\tr^*(\mathcal{B}^+_x)\rightarrow \tr^*(\mathcal{B}^+_{\sigma(x)})$, yielding a linear cocycle over the shift $\sigma$, which we call the \textbf{trace cocycle}. The Lyapunov spectrum of this cocycle, that is, growth rate of cotrace vectors under the trace cocycle, is what controls the statistical properties of the tilings.

Let $\mathcal{L}(\Omega_x)$ denote the set of Lipschitz functions on $\Omega_x$. For any Oseledets-regular $x$, that is, for any $x$ for which the conclusion of the Oseledets theorem holds (see \S \ref{subsec:TrCocycle}), there is a map $i^+_x:\mathcal{L}(\Omega_x)\rightarrow LF(\mathcal{B}^+_x)$ (see \S \ref{sec:erg}) and we denote by $[a_f] = i^+_x(f)$ the image of $f\in\mathcal{L}(\Omega_x)$ through this map. Before stating the first theorem, some notation is needed. 
For a set $B\subset\mathbb{R}^d$ we denote by $T\cdot B$ the scaling $(T\cdot \mathrm{Id})B$. A good Lipschitz domain is defined in \S \ref{subsec:domains}, but for now it suffices to say that it is a set whose boundary is not too complicated.
\begin{theorem}
\label{thm:main}
Let $\mathcal{F}$ be a finite family of uniformly expanding and compatible substitution rules on a finite set of prototiles $\{t_1,\dots, t_M\}$ with $X_\mathcal{F}\subset \Sigma_N$ parametrizing the possible tiling spaces and $\mu$ a minimal, $\sigma$-invariant ergodic Borel proability measure.

There exist Lyapunov exponents $\lambda_1^+>\lambda_2^+\geq \cdots\geq  \lambda_{d^+_\mu}^+ >0 $ (depending on $\mathcal{F}$ and $\mu$) such that for $\mu$-almost every $x$ there are traces $\tau_1^+,\dots, \tau_{d_\mu^+}^+\in \tr(\mathcal{B}^+_x)$ such that if $f\in\mathcal{L}(\Omega_x)$ satisfies $\tau_i^+([a_f])=0$ for all $i<j$ for some $j\leq d^+_\mu$ but $\tau^+_j([a_f])\neq 0$, $B$ a good Lipschitz domain and $\mathcal{T}\in\Omega_x$ then for every $\varepsilon>0$ there exists a set $B_\varepsilon$ which is $\varepsilon$-close to $B$ in the Hausdorff metric, a sequence $T_k\rightarrow \infty$ and a convergent sequence of vectors $\tau_k$ such that:
  \begin{equation}
    \label{eqn:uppBnd2}
    \limsup_{k\rightarrow \infty}\frac{\log \left| \displaystyle \int_{T_k(B_\varepsilon+\tau_k)}f\circ \varphi_t(\mathcal{T})\, dt\right|}{\log T_k}= d\frac{\lambda_j^+}{\lambda^+_1}.
    \end{equation}
  If in addition $d\lambda_j^+\geq(d-1)\lambda_1^+$ then:
  \begin{equation}
    \label{eqn:uppBnd1}
    \limsup_{T\rightarrow \infty}\frac{\log \left| \displaystyle \int_{T\cdot B}f\circ \varphi_t(\mathcal{T})\, dt\right|}{\log T}\leq d\frac{\lambda_j^+}{\lambda^+_1}.
    \end{equation}
\end{theorem}
\begin{remark}
  Some remarks:
  \begin{enumerate}
  \item The case of self-similar tilings, tilings which are constructed from a single substitution rule, correspond to tiling spaces $\Omega_x$ for periodic points $x$ under the shift $\sigma:\Sigma_N\rightarrow \Sigma_N$. In other words, tiling spaces for self-similar tilings correspond to the typical points of finitely-supported invariant measures on $\Sigma_N$ (assuming they are minimal measures). Studies of deviations of ergodic integrals for such types of systems have been done elsewhere \cite{sadun:exact, bufetov-solomyak:tilings, ST:SA}. Therefore, what is new here are the results for tiling spaces which do not come from self-similar tilings, that is, tilings which come from tiling spaces $\Omega_x$ for $x$ a typical point of a $\sigma$-invariant, ergodic measure satisfying the hypotheses of the Theorem which is not finitely supported. There is a continuum of examples in \S \ref{sec:experiments}.
    \item The spectral gap $\lambda_1^+>\lambda_2^+$ is a consequence of the recent general spectral gap result of Horan \cite[Corollary 2.19]{horan:lyapunov}.
  \end{enumerate}
\end{remark}

There is a particular type of tiling space, called a solenoid, which satisfies a type of bound known as Denjoy-Koksma inequality (the trace space is trivial for solenoids so Theorem \ref{thm:main} does not yield any information). The solenoid construction here is dependent on a family of substitution rules given by a sequence of positive integers $\bar{q} = (q_1,q_2,\dots)$ each one greater than 1. There is also a concept of function of bounded variation on the solenoid $\Omega_{\bar{q}}$ and the space of all functions of bounded variation on $\Omega_{\bar{q}}$ is denoted by $\mathrm{BV}(\Omega_{\bar{q}})$ (see \S \ref{sec:DK}). For $\bar{q}\in \mathbb{N}^\mathbb{N}$ denote by $q_{(n)} = q_1q_2\cdots q_n$, and let $\mu$ be the unique invariant mesaure on $\Omega_{\bar{q}}$.
\begin{theorem}
  \label{thm:DK}
  Let $\Omega_{\bar{q}}$ be a $d$-dimensional solenoid. Then for any $f\in \mathrm{BV}(\Omega_{\bar{q}})$ and $p\in \Omega_{\bar{q}}$,
  $$\left| \int_{[0,q_{(n)}]^d} f\circ \varphi^+_s(p)\, ds - q_{(n)}^d\int_{\Omega_{\bar{q}}} f\, d\mu \right|\leq \mathrm{Var}(f)$$
  for all $n>0$.
\end{theorem}
The Denjoy-Koksma inequality was first proved for irrational circle rotations by Herman \cite[Th\'eor\`eme VI.3.1]{herman:cercle}. Theorem \ref{thm:DK} here is the first instance of this type of inequality for higher rank systems.

Let $\mathcal{T}$ be a repetitive tiling with finitely many prototiles. Consider the Delone set $\Lambda_\mathcal{T}\subset\mathbb{R}^d$ obtained by puncturing every prototile in its interior and forming $\Lambda_\mathcal{T}$ as the union of all the corresponding punctures on tiles of $\mathcal{T}$ which correspond to punctures of the prototiles. There is a class of operators on $\ell^2(\Lambda_\mathcal{T})$, called the \textbf{Lipschitz operators of finite range}, denoted by $\mathcal{LA}_x^{fin}$. These operators are defined in \S \ref{sec:schrod}, but what is relevant here is that they contain operators of interest in mathematical physics, namely self-adjoint operators of the form $H = \triangle + V$, where $\triangle$ is a Laplacian-type operator and $V$ is any potential reflecting the aperiodic and repetitive nature of all tilings in $\Omega_x$ (see the footnote on page \pageref{lem:lipMap}). There types of operators sometime go under the name of \textbf{random Schr\"odinger operators}.

For $\mathcal{T}\in\Omega_x$ and $A\in \mathcal{LA}_x^{fin}$ a self-adjoint operator, we have the operator $A_\mathcal{T}$ acting on $\ell^2(\Lambda_{\mathcal{T}})$ and this assignment is equivariant with respect to the $\mathbb{R}^d$ action on $\Omega_x$. Denote by $A_\mathcal{T}|_{B}$ the restriction of $A_\mathcal{T}$ to the finite dimensional subspace $\ell^2(\Lambda_{\mathcal{T}}\cap B)\subset \ell^2(\Lambda_{\mathcal{T}})$ defined by $\Lambda_{\mathcal{T}}\cap B$. For $E\in\mathbb{R}$ and $T>0$ denote by
$$n_T^A(E)  := \# \{\mbox{eigenvalues of $A_\mathcal{T}|_{B_T}$ which are $\leq E$ } \}.$$
Assuming $\mathcal{T}$ is repetitive, has finite local complexity and uniform patch frequency, that is, $\mathcal{T}$ corresponds to a minimal and uniquely ergodic system, the function
$$E\mapsto \lim_{T\rightarrow \infty}\frac{n_T^A(E)}{\mathrm{Vol}(B_T)}$$
is the distribution of a measure $\rho_A$ (independent of $\mathcal{T}$ in the tiling space), called the \textbf{integrated density of states} \cite{LS:IDS}, satisfying
\begin{equation}
  \label{eqn:shubin}
  \rho_A(\varphi)= \tau(\varphi(A)):=\lim_{T\rightarrow \infty}\frac{\mathrm{tr}(\varphi(A|_{B_T}))}{\mathrm{Vol}(B_T)},
\end{equation}
where $\mathrm{tr}(\varphi(A|_{B_T}))$ is the unique (non-normalized) trace of the finite dimensional operator $\varphi(A|_{B_T})$, and for any continuous $\varphi$. This is the \textbf{Shubin-Bellissard trace formula}. It should be emphasized that the fact that the limit in (\ref{eqn:shubin}) is a trace is not trivial; see \cite[Lemma 3.4]{LS:algebras}. For a thorough introduction to the study of spectral properties of Schr\"odinger operators emerging from quasicrystals, see \cite[\S 3]{DEG:survey}.

The question addressed in \cite{ST:traces} was: what can be said about the convergence in (\ref{eqn:shubin})? In other words, is there a $\lambda\in(0,d)$ such that
$$\left|\mathrm{tr}(\varphi(A|_{B_T})) - \mathrm{Vol}(B_T)\tau(\varphi(A))\right| \leq CT^\lambda $$
for some $C>0$ and all $T>1$? The main result of \cite{ST:traces} showed that if the tiling or Delone set had a self-affine structure, then yes, error rates for the Shubin-Bellissard trace formula can be computed, and that they can can be computed with the help of other traces.

The second main result of this paper is a generalization of the main result of \cite{ST:traces} and answering this question in the case of random susbtitution tilings. Not only are the error rates for the convergence in (\ref{eqn:shubin}) computed, but the traces responsible for them are related to the traces defined on the LF algebras $LF(\mathcal{B}_x^+)$ and the error rates are defined by the Lyapunov spectrum of the trace cocycle from Theorem \ref{thm:main}. More precisely, in \S \ref{sec:schrod} for almost every $x\in X_\mathcal{F}$ we define a map $\Upsilon_x:\mathcal{LA}_x^{fin}\rightarrow LF(\mathcal{B}^+_x)$ and define functionals $\tau_i':= \Upsilon_x^*\tau_i^+$ by pulling back some of the traces in $\tr(\mathcal{B}^+_x)$. Whether or not $\tau_i'$ is a trace on $\tr(\mathcal{LA}_x^{fin})$ is dependent on the Lyapunov exponent $\lambda_i^+$ (see Proposition \ref{prop:pullTrace}). The following is a consequence of Theorem \ref{thm:main}.
\begin{theorem}
  \label{thm:opDevs}
  Let $\mathcal{F}$ be a finite family of uniformly expanding and compatible substitution rules on a finite set of prototiles $\{t_1,\dots, t_M\}$ with $X_\mathcal{F}\subset \Sigma_N$ parametrizing the possible tiling spaces and $\mu$ a minimal, ergodic, $\sigma$-invariant Borel ergodic proability measure.

  There exist Lyapunov exponents $\lambda_1^+>\lambda_2^+\geq \cdots\geq  \lambda_{d^+_\mu}^+>0 $ (depending on $\mathcal{F}$ and $\mu$) such that for $\mu$-almost every $x$ there are traces $\tau_1,\dots, \tau_{d_\mu^+}\in \tr(\mathcal{B}_x^+)$ such that if $A\in\mathcal{LA}_x^{fin}$ satisfies $\tau_i(\Upsilon_x(A))=0$ for all $i<j$ for some $j\leq d^+_\mu$ but $\tau_j(\Upsilon_x(A))\neq 0$, for $B$ a good Lipschitz domain, for every $\varepsilon>0$ there exists a set $B_\varepsilon$ which is $\varepsilon$-close to $B$ in the Hausdorff metric, a sequence $T_k\rightarrow \infty$ and a convergent sequence of vectors $\tau_k$ such that:
  $$\limsup_{k\rightarrow \infty}\frac{\log |\mathrm{tr}(A_\mathcal{T}|_{T_k\cdot( B_{\varepsilon}+\tau_k)})|}{\log T_k} = d\frac{\lambda^+_r}{\lambda_1^+}.$$
  If, in addition, $d\lambda_j^+> (d-1)\lambda_1^+$, then $\tau_j' = \Upsilon_x^*\tau_j$ is a trace and
$$\limsup_{T\rightarrow \infty}\frac{\log |\mathrm{tr}(A_\mathcal{T}|_{T\cdot B})|}{\log T} \leq d \frac{\lambda^+_j}{\lambda_1^+}.$$  
  \end{theorem}
\begin{remark}
  \label{rem:2}
  Some remarks:
  \begin{enumerate}
  \item These estimates give rates of convergence for the integrated density of states in (\ref{eqn:shubin}) for random Schr\"odinger operators as explained in the paragraphs following (\ref{eqn:shubin}). For example, if under the hypotheses of the theorem the top two Lyapunov exponents satisfy $\lambda^+_2>(d-1)\lambda_1^+/d$, then for any $\varepsilon>0$,
    $$\left|\mathrm{tr}(A|_{B_T}) - \mathrm{Vol}(B_T)\tau(A)\right| \leq C_\varepsilon T^{d\frac{\lambda^+_2}{\lambda_1^+}+\varepsilon} $$
    for some $C_\varepsilon>0$ and all $T>1$.
  \item Just like many of the traces on $LF(\mathcal{B}^+_x)$, a dense subalgebra of the C$^*$-algebra $AF(\mathcal{B}_x^+)$, do not extend to the full C$^*$-algebra, the auxiliary traces which describe the error rates in the convergence of the integrated density of states \emph{do not extend to traces on any C$^*$-algebra}. Thus what is important here is not the C$^*$-algebra of Schr\"odinger operators but a dense $*$-subalgebra consisting of ``smooth'' operators, which in this case is $\mathcal{LA}_x^{fin}$.
  \item This statement has no immediate relation to any statement about gap labelling (see \cite{kellendonk:gap} for background).
  \item  It is unclear to me what physical interpretations the traces $\tau_i'$ in Theorem \ref{thm:opDevs} have.
  \end{enumerate}
\end{remark}

This paper is organized as follows: in sections \ref{sec:tilings} and \ref{sec:brat} we review the essential definitions related to tilings and Bratteli diagrams, and how one can construct tiling spaces using Bratteli diagrams. These sections cover background material, borrowing some results from \cite{ST:random}. Section \ref{sec:interlude} is an interlude which illustrates the constructions using the example of half-hexagons in Figure \ref{fig:HalfHexes}. Section \ref{sec:LF} covers locally finite subalgebras of AF algebras and their traces. It is in this section that the trace cocycle is introduced and some basic properties are derived. Section \ref{sec:erg} is devoted to the study of ergodic integrals for Lipschitz functions on tiling spaces using the trace cocycle. Section \ref{sec:DK} proves the Denjoy-Koksma inequality for general solenoids. Finally, section \ref{sec:schrod} covers the application of the main theorem on deviations of ergodic averages to traces on random Schr\"odinger operators. The last section of the paper, \S \ref{sec:experiments}, shows some experimental results for the easiest non-trivial results I could come up with using half hexagons. It strongly suggests that in this case the Lyapunov spectrum is non-singular but does have multiplicities.

\begin{ack}
  I am deeply grateful to Lorenzo Sadun who pointed out a mistake in an earlier version of this paper, and to Dan Rust for helpful discussions, especially bringing \cite{GL:random} to my attention. I am also grateful to an anonymous referee for suggestions which made the exposition of the paper much better.   This work was supported by NSF grant DMS-1665100.
\end{ack}
\section{Tilings} 
\label{sec:tilings}
This section introduces the basic concepts in the theory of tilings. For a more thorough overview, see e.g. \cite{BaakeGrimm:book1} and \cite{sadun:book}.

A \textbf{tile} $t$ is a compact, connected subset of $\mathbb{R}^d$. Here it will always be assumed that the boundary $\partial t$ of a tile has finite $d-1$ dimensional measure. A \textbf{tiling} $\mathcal{T}$ of $\mathbb{R}^d$ is a cover of $\mathbb{R}^d$ by tiles, where two different tiles may only intersect along their boundaries. Here we will consider only cases where the tilings are formed by a finite set of prototiles $\{t_1,\dots, t_M\}$. That is, every tile $t\in \mathcal{T}$ is a translated copy of $t_i$ for some $i$. A \textbf{patch} of $\mathcal{T}$ is a finite connected union of tiles of $\mathcal{T}$. A tiling $\mathcal{T}$ is called repetitive if for any patch $\mathcal{P}$ there exists an $R_\mathcal{P}>0$ such that any ball of radius $R_\mathcal{P}$ contains a translated copy of $\mathcal{P}$ in it. For any set $A\subset \mathbb{R}^d$, denote by
$$\mathcal{O}^-_\mathcal{T}(A) = \mbox{ largest patch $\mathcal{P}$ of $\mathcal{T}$ completely contained in $A$}.$$
A tiling has \textbf{finite local complexity} if for each $R>0$ there exists a finite collection of patches $\mathcal{P}_1,\dots,\mathcal{P}_{N(R)}$ such that for any $x\in\mathbb{R}^d$ the patch $\mathcal{O}^-_{\mathcal{T}}(B_R(x))$ is a translated copy of one of the patches $\mathcal{P}_i$. 

      A \textbf{substitution rule} $\mathcal{F}$ on a finite set of prototiles $\{t_1,\dots, t_M\}$ is a rule which allows to express each prototile $t_{n_i}$ in a subset\footnote{Note that this differs from the traditional definition of a subtitution rule in that traditionally it is \emph{all} prototiles which are subdivided, whereas here one is allowed to only consider a subset of them and ignore the rest.} $\{t_{n_1},\dots ,t_{n_{\mathcal{F}}}\}\subset \{t_1,\dots, t_M\}$ as the finite union of scaled copies of some of the prototiles. More precisely, suppose we identify each prototile $t_i$ with a subset of $\mathbb{R}^d$, and we assume without loss of generality that this subset contains the origin in its interior. Then a substitution rule consists of a collection of scaling maps (also called graph iterated function systems) $\mathcal{F} = \{f_{i,j,k}:\mathbb{R}^d\rightarrow \mathbb{R}^d\}$ with $i,j = 1,\dots, M$, $k = 1,\dots, r(i,j)$ such that
      \begin{equation}
        \label{eqn:GIFS}
        t_{n_i} = \bigcup_{j=1}^M \bigcup_{k = 1}^{r(n_i,j)} f_{n_i,j,k}(t_j)
        \end{equation}
      and if for any $i$ any two maps $f_{i,j,k}$ and $f_{i,j',k'}$ have $f_{i,j,k}(t_j)\cap f_{i,j',k'}(t_j)\neq \varnothing$, then the intersection happens along the boundary of the images. In other words, each $t_{n_i}$ can be tiled by scaled copies of the prototiles $t_i$. The number $n(i,j)$ is the number of copies of a rescaled copy of the prototile $t_j$ is placed in $t_i$ when subdividing. As such, $f_{i,j,1}$ exists only if there is a rescaled copy of $t_j$ found when substituting the prototile $t_i$. The reader who has not seen a substitution rule defined as in (\ref{eqn:GIFS}) is invited to \S \ref{sec:interlude}, where the example in Figure \ref{fig:HalfHexes} is illustrated from the point of view of (\ref{eqn:GIFS}).

      A substitution is \textbf{uniformly expanding} if all maps  $f_{i,j,k}\in \mathcal{F}$ are of the form $f_{i,j,k}(x) = rx + \tau_{i,j,k}$ for some $r\in(0,1)$ and $\tau_{i,j,k}\in\mathbb{R}^d$. In this case, if $r$ is the contracting factor, $r^{-1}\cdot t_{n_i}$ the the union of prototiles, and the rescaling of (\ref{eqn:GIFS}) as
      \begin{equation}
        \label{eqn:GIFS2}
        r^{-1}\cdot t_{n_i} = \bigcup_{j=1}^M \bigcup_{k = 1}^{r(n_i,j)} r^{-1}\cdot f_{n_i,j,k}(t_j)
      \end{equation}
      is a \textbf{substitution and expansion rule} (Figure \ref{fig:HalfHexes} gives an example of two such rules, one with contraction 1/2 and the other with contraction 1/4). In defining a substitution rule which uniformly expanding, one is implicitly defining a substitution and expansion rule by (\ref{eqn:GIFS2}). 
      
      We will transform tilings by two types of operations: translations and deformations. Let $\mathcal{T}$ be a tiling of $\mathbb{R}^d$ and $\tau\in \mathbb{R}^d$. Then the tiling $\varphi_\tau(\mathcal{T}) := \mathcal{T} - \tau$ is the tiling of $\mathbb{R}^d$ obtained by translating each tile of $\mathcal{T}$ by the vector $\tau\in \mathbb{R}^d$. This is the \textbf{translation} of $\mathcal{T}$ by $\tau$.
      
      All the tilings which will be considered in this paper will have finite local complexity, so it will be assumed from now on. If tiling $\mathcal{T}$ has finite local complexity, then define a metric on the set of all translates of $\mathcal{T}$ by
      \begin{equation}
        \label{eqn:metric1}
        d\left(\mathcal{T},\varphi_t(\mathcal{T})\right) = \min \left\{1,\bar{d}(\mathcal{T},\varphi_t(\mathcal{T}))\right\},
        \end{equation}
      where
      \begin{equation}
        \label{eqn:metric2}
        \bar{d}\left(\mathcal{T},\varphi_t(\mathcal{T})\right) = \inf \left\{\varepsilon > 0:\mathcal{O}_\mathcal{T}^-\left(B_{\frac{1}{\varepsilon}}\right) \bowtie \mathcal{O}_{\varphi_{t+s}(\mathcal{T})}^-\left(B_{\frac{1}{\varepsilon}}\right) \mbox{ for some }\|s\|\leq \varepsilon\right\},
      \end{equation}
      where $\mathcal{P}_1\bowtie\mathcal{P}_2$ denotes the equivalence of patches $\mathcal{P}_1$ and $\mathcal{P}_2$ by a translation. In words: two tilings are close if they agree on a large ball around the origin up to a small translation. That this is a metric for tilings of finite local complexity is standard; see \cite[\S 5.4]{BaakeGrimm:book1}. The \textbf{tiling space} of $\mathcal{T}$ is defined as the metric completion of all translates of $\mathcal{T}$ with respect to the metric above:
      $$\Omega_\mathcal{T} = \overline{\{\varphi_t(\mathcal{T}):t\in\mathbb{R}^d\}}.$$
      
      There is a natural action of $\mathbb{R}^d$ on $\Omega_\mathcal{T}$ by translation, $\varphi_t:\mathcal{T}'\mapsto \varphi_t(\mathcal{T}')$. The action being minimal is equivalent to $\mathcal{T}$ being repetitive. As such, if $\mathcal{T}$ is repetitive then for any two $\mathcal{T}_1,\mathcal{T}_2\in\Omega_\mathcal{T}$ we have that $\Omega_{\mathcal{T}_1} = \Omega_{\mathcal{T}_2}$.

      Suppose $\mathcal{T}$ is a tiling of $\mathbb{R}^d$ by a finite collection of prototiles. That is, there is a finite set of tiles $\{t_1,\dots, t_M\}$ such that every tile $t\in \mathcal{T}$ is translation equivalent to $t_i$ for some $i$. For each $i$, pick a distinguished point in the interior of the prototile $t_i$, and then distinguish a point in the interior of each of the tiles in $\mathcal{T}$ by the translation equivalence between the tiles and prototiles. The \textbf{canonical transversal} $\mho_\mathcal{T}\subset \Omega_\mathcal{T}$ is the set
      $$\mho_\mathcal{T}:= \{\mathcal{T}'\in\Omega_\mathcal{T}:\mbox{ the distinguished point in the tile $t\in\mathcal{T}'$ containing the origin is the origin}\}.$$
If $\mathcal{T}$ is repetitive then $\mho_\mathcal{T}$ is a true transversal for the action of $\mathbb{R}^d$ on $\Omega_\mathcal{T}$ since it intersects every orbit.
            
Let $\mathcal{P}$ be a patch of $\mathcal{T}$ and $t\in\mathcal{P}$ a choice of one of the tiles in that patch. The the $(\mathcal{P},t)$-\textbf{cylinder set} is defined as
\begin{equation}
  \label{eqn:transvSets}
  \mathcal{C}_{\mathcal{P},t} = \{\mathcal{T}'\in\Omega_\mathcal{T}:\mathcal{P}\mbox{ is a patch in $\mathcal{T}'$ and the distinguished point in $t\in\mathcal{P}$ is the origin}\}
  \end{equation}
and note that this is a subset of $\mho_\mathcal{T}$. In fact, the topology of $\mho_\mathcal{T}$ is generated by cylinder sets of the form $\mathcal{C}_{\mathcal{P},t}$ and it has the structure of a Cantor set whenever $\mathcal{T}$ has finite local complexity. Note that for two tiles $t,t'\in\mathcal{P}$ (not necessarily of the same type) there exists a vector $\tau = \tau(\mathcal{P},t,t')$ such that $\varphi_\tau(\mathcal{C}_{\mathcal{P},t}) = \mathcal{C}_{\mathcal{P},t'}$.

For a patch $\mathcal{P}$ with a distinguished point in its interior, a tile $t\in\mathcal{P}$ and $\varepsilon >0$, the $(\mathcal{P},t,\varepsilon)$-cylinder set is the set
\begin{equation}
  \label{eqn:openSet}
\mathcal{C}_{\mathcal{P},t,\varepsilon} = \bigcup_{\|t\|<\varepsilon}\{\varphi_t(\mathcal{T}'):\mathcal{T}'\in\mathcal{C}_{\mathcal{P},t}\}\subset \Omega_\mathcal{T}.
  \end{equation}
For a repetitive  $\mathcal{T}$ of finite local complexity the topology of $\Omega_\mathcal{T}$ is then generated by cylinder sets of the form $\mathcal{C}_{\mathcal{P},t,\varepsilon}$ with $\mathcal{P}$ being any patch in $\mathcal{P}$ and $\varepsilon>0$ being arbitrarily small. This gives $\Omega_\mathcal{T}$ a local product structure of $B_\varepsilon(0)\times \mathcal{C}$, where $B_\varepsilon(0)\subset \mathbb{R}^d$ is the open ball of radius $\varepsilon$ and $\mathcal{C}$ is a Cantor set.

Let $\mathcal{T}$ be a repetitive tiling of finite local complexity. Given a patch $\mathcal{P}\subset \mathcal{T}$ and set $B\subset \mathbb{R}^d$ let $L_\mathcal{T}(P,B)$ be the number of copies of $\mathcal{P}$ completely contained inside of $B$. Then
      $$\mathrm{freq}_\mathcal{T}(P) = \lim_{T\rightarrow \infty}\frac{L_\mathcal{T}(P,B_T)}{\mathrm{Vol}(B_T)},$$
      when it exists, is the asymptotic patch frequency of $\mathcal{P}$ in $\mathcal{T}$. For the purposes of this paper, without loss of generality, it can be assumed that this limit always exists since it will be well-defined for all tilings considered here. By (\ref{eqn:transvSets}) this gives a family of Borel measures on $\mho_\mathcal{T}$ parametrized by $\Omega_\mathcal{T}$ which are invariant under the holonomies $\tau(\mathcal{P},t,t')$. In other words, we have a function $\nu:\Omega_\mathcal{T}\times \mathcal{B}(\mho_\mathcal{T})\rightarrow \mathbb{R}$, where $\mathcal{B}(\mho_\mathcal{T})$ is the Borel $\sigma$-algebra of $\mho_\mathcal{T}$, with $\nu(\mathcal{T}',\mathcal{P}) = \mathrm{freq}_{\mathcal{T}'}(\mathcal{P})$ for any patch $\mathcal{P}$. The action of $\mathbb{R}^d$ on $\Omega_\mathcal{T}$ is \textbf{uniquely ergodic} if $\nu$ does not depend on the first coordinate, that is, $\mathrm{freq}_\mathcal{T}(\mathcal{P})$ is independent of $\mathcal{T}$. This will be the typical case in this paper; see \cite[\S 3]{solomyak:SS} for further details about frequencies.

      Given that the measures $\nu_\mathcal{T}:=\nu(\mathcal{T},\cdot)$ are holonomy-invariant, by the local product structure of $\Omega_\mathcal{T}$, they define $\mathbb{R}^d$-invariant measures on $\Omega_\mathcal{T}$ which are locally of the form $\mu_\mathcal{T} = \mathrm{Leb}\times \nu$, where $\nu$ is defined by the restriction the frequency measure $\nu_\mathcal{T}$ on the Cantor set defined by the patch $\mathcal{P}$. Whenever $\varphi_s:\Omega_\mathcal{T}\rightarrow \Omega_\mathcal{T}$ is uniquely ergodic we will denote by $\mu$ the unique invariant measure.
      \subsection{Lipschitz domains}
      \label{subsec:domains}
      This subsection introduces Lipschitz domains, which are types of subsets of $\mathbb{R}^d$ whose boundaries are well-behaved, making them useful sets over which to integrate functions. Let $\mathcal{H}^m$ denote the $m$-dimensional Hausdorff measure.
      \begin{definition}
        A set $E\subset \mathbb{R}^d$ is called \emph{$m$-rectifiable} if there exist Lipschitz maps $f_i: \mathbb{R}^m\rightarrow\mathbb{R}^d$, $i = 1,2,\dots$ such that
        $$\mathcal{H}^m\left(  E\backslash \bigcup_{i\geq 0} f_i(\mathbb{R}^m)   \right) = 0.$$
      \end{definition}
\begin{definition}
  A \emph{Lipschitz domain} $A\subset\mathbb{R}^d$ is an open, bounded subset of $\mathbb{R}^d$ for which there exist finitely many Lipschitz maps $f_i:\mathbb{R}^{d-1}\rightarrow \mathbb{R}^d$, $i = 1,\dots, L$ such that
  $$\mathcal{H}^{d-1}\left(  \partial A \backslash \bigcup_{i=1}^L f_i(\mathbb{R}^{d-1})   \right) = 0.$$
\end{definition}
Lipschitz domains have $d-1$-rectifiable boundaries.
\begin{definition}
A subset $A\subset \mathbb{R}^d$ is a \textbf{good} Lipschitz domain if it is a Lipschitz domain and $\mathcal{H}^{d-1}(\partial A)<\infty$.
\end{definition}
\section{Bratteli diagrams and tilings}
\label{sec:brat}
      A \textbf{Bratteli diagram} $\mathcal{B} = (\mathcal{V},\mathcal{E})$ is a bi-infinite directed graph partitioned such that
      $$\mathcal{V} = \bigsqcup_{k\in\mathbb{Z}} \mathcal{V}_k\hspace{1in}\mbox{ and }\hspace{1in}\mathcal{E} = \bigsqcup_{k\in \mathbb{Z}\backslash \{0\}} \mathcal{E}_k$$
      with maps $r,s:\mathcal{E}\rightarrow \mathcal{V}$ satisfying $r(\mathcal{E}_{k}) = \mathcal{V}_{k}$ $s(\mathcal{E}_{k}) = \mathcal{V}_{k-1}$ if $k>0$ and $r(\mathcal{E}_{k}) = \mathcal{V}_{k+1}$ $s(\mathcal{E}_{k}) = \mathcal{V}_{k}$ if $k<0$, and with $r^{-1}(v)\neq \varnothing$ and $s^{-1}(v)\neq \varnothing$ for all $v\in\mathcal{V}$. We assume that $|\mathcal{V}_k|$ and $|\mathcal{E}_k|$ are finite for every $k$.
      \begin{remark}
        The above definition is not the usual definition of a Bratteli diagrams, as usually their edge and vertex sets are indexed by $\mathbb{N}$. One of the reasons to index the edge set $\mathcal{E}$ through $\bar{\mathbb{Z}}$ instead of $\mathbb{Z}$ is that it makes labelling choices when drawing them less awkward. The ones considered here are technically bi-infinite diagrams and the notational conventions of \cite{LT} for bi-infinite Bratteli diagrams will be followed. There are two other advantages of using bi-infinite diagrams rather than the traditional diagrams indexed by $\mathbb{N}$; see the first paragraph of \S \ref{subsec:renorm} for details.
      \end{remark}
      The \textbf{positive part} $\mathcal{B}^+$ of $\mathcal{B}$ is the Bratteli diagram $\mathcal{B}^+ = (\mathcal{V}^+,\mathcal{E}^+)$ defined by the restriction to the non-negative indices of the data of $\mathcal{B}$. The \textbf{negative part} $\mathcal{B}^-$ is similarly defined.

      A \textbf{path} in $\mathcal{B}$ is a finite collection of edges $\bar{e} = (e_\ell,\dots , e_m)$ such that $e_i\in \mathcal{E}_i $ and $r(e_i) = s(e_{i+1})$ for all $i\in \{\ell,\dots, m-1\}$. As such the domain of the range and source maps can be extended to all finite paths by setting $s(\bar{e}) = s(e_\ell)$ and $r(\bar{e}) = r(e_m)$. Let $\mathcal{E}_{\ell,m}$ be the set of all paths starting $\mathcal{V}_\ell$ to $\mathcal{V}_m$, that is, finite paths $\bar{e}$ with both $s(\bar{e})\in \mathcal{V}_\ell$ and $r(\bar{e})\in\mathcal{V}_m$. We can extend this to infinite paths: let $X^+_\mathcal{B} = \mathcal{E}_{0,\infty}$ be the set of infinite paths starting at $\mathcal{V}_0$ and $X^-_\mathcal{B} = \mathcal{E}_{-\infty,0}$ be the set of infinite paths ending at $\mathcal{V}_0$. The set $X^+_\mathcal{B}$ can be topologized by cylinder sets of the form
      \begin{equation}
        \label{eqn:cylinderBrat}
        C(\bar{e}) = \{\bar{p}\in X^+_\mathcal{B}:(p_\ell,\dots, p_m) = (e_\ell,\dots, e_m)\},
      \end{equation}
      for some finite path $\bar{e}\in \mathcal{E}_{\ell,m}$ with $0\leq \ell < m$. The set $X^-_\mathcal{B}$ is similarly topologized and as such the spaces $X^\pm_\mathcal{B}$, when the number of vertices at every level is uniformly bounded\footnote{As they will be in the diagrams appearing in this paper.}, are compact metric spaces which are Cantor sets. The space of all bi-infinite paths on $\mathcal{B}$ is then
      $$X_\mathcal{B}\subset X^-_\mathcal{B}\times X^+_\mathcal{B}$$
      and it inherits the subspace topology.

      Two paths $p,p'\in X^+_\mathcal{B}$ are \textbf{tail-equivalent} if there is an $N>0$ such that $p_i = p'_i$ for all $i>N$ and this is an equivalence relation, where we deonte classes by $[\bar{e}]$. A minimal component of $X^+_\mathcal{B}$ is a subset of the form $\overline{[\bar{e}]}$. A Bratteli diagram $\mathcal{B}$ is \textbf{minimal} if $\overline{[\bar{e}]} = X^+_\mathcal{B}$ for all $\bar{e}\in X^+_\mathcal{B}$ or, in other words, when there is only one minimal component. A measure $\mu$ on $X^+_\mathcal{B}$ is \textbf{invariant under the tail equivalence relation} if for any $N$ and paths $p_1,p_2\in \mathcal{E}_{0,N}$ with $r(p_1) = r(p_2)$ we have that $\mu(C(p_1)) = \mu(C(p_2))$.
      
      \subsection{Tilings from diagrams}
      \label{sec:blow}
      Here we recall the tiling construction from \cite{ST:random}.  Let $\{t_1,\dots, t_M\}$ be a set of prototiles and suppose that they admit $N$ substitution rules $\mathcal{F}_1,\dots, \mathcal{F}_N$.   Given a collection $\mathcal{F} = \{\mathcal{F}_1,\dots, \mathcal{F}_N\}$ of substitution rules, we want to parametrize all possible tilings we can obtain by different combinations of substitutions. As such, the space that organizes all of these combinations is a $\sigma$-invariant, closed subset  $X_\mathcal{F} \subset \Sigma_N = \{1,\dots, N\}^{\bar{\mathbb{Z}}}$ of the $N$-shift, where $\bar{\mathbb{Z}} := \mathbb{Z}-\{0\}$, inheriting the order from $\mathbb{Z}$. In this section, a procedure is described for constructing from any $x\in X_\mathcal{F}$ a Bratteli diagram $\mathcal{B}^+_x$ and a construction assigning paths $\bar{e}\in X^+_\mathcal{B}$ a tiling $\mathcal{T}_{\bar{e}}$. If all the substitution rules $\mathcal{F}_1,\dots, \mathcal{F}_N$ involve all prototiles (see footnote on page \pageref{eqn:GIFS}), then $X_\mathcal{F} = \Sigma_N$. However, if one or more of the substitution rules do not involve all prototiles, then there may be restrictions as to how one can compose them, leading to a strict subset $X_\mathcal{F}\subset \Sigma_N$ which would be $\sigma$-invariant (a subshift). The reader should always keep in mind the case where all uniformly expanding substitution rules $\mathcal{F}_1,\dots, \mathcal{F}_N$ involve all prototiles (and so $X_\mathcal{F} = \Sigma_N$); the other cases are not usually common in the literature, but they can be handled with the machinery of this paper.

      Pick $x =(x^-,x^+) = (\dots, x_{-2},x_{-1},x_1,x_2,\dots)\in X_\mathcal{F} $. We will start by defining the positive part $\mathcal{B}^+_x$ of the Bratteli diagram $\mathcal{B}_x$. For $k\geq 0$, $\mathcal{B}_x^+$ will have $|\mathcal{V}_k|$ be the number of tiles used in the substitution $\mathcal{F}_{x_{k+1}}$, that is, not the number of tiles which are tiled by the rule $\mathcal{F}_{x_k}$, but the number of different tiles used in that substitution rule\footnote{If all substitution rules $\mathcal{F}_1,\dots, \mathcal{F}_N$ involve all prototiles, then $|\mathcal{V}_k| = N$ for all $k$. Again, see footnote on page \pageref{eqn:GIFS}.}. The vertices are ordered at each level so that $v_i\in \mathcal{V}_k^+$ is identified with $t_{n_i}$ for every $k$. Now, starting with $k=1$, consider the substitution rule $\mathcal{F}_{x_k}$. Then for $v_j\in \mathcal{V}_{k-1}^+$ and $v_i\in \mathcal{V}_k^+$ there are $r(i,j)$ edges from $v_j$ to $v_i$, and we identify the corresponding map $f_{i,j,k}$ with the appropriate edge $e\in \mathcal{E}_k^+$ and denote it by $f_e$. Since the maps $f_e$ are contacting they are of the form $f_e(x) = \theta_ex + \tau_e$ for some $\theta_e\leq 1$. This notation extends to finite paths $\bar{e}\in \mathcal{E}_{0,k}$ by $f_{\bar{e}} = f_{e_k}\circ \cdots \circ f_{e_1}$.

      Let $\bar{e}\in X^+_\mathcal{B}$ and denote by $\bar{e}|_k$ the truncation of $\bar{e}$ after its $k^{th}$ edge, that is, $\bar{e}|_k\in \mathcal{E}^+_{0,k}$. The \textbf{$k^{th}$ approximant $\mathcal{P}_k(\bar{e})$} is the set
      \begin{equation}
        \label{eqn:approx}
        \mathcal{P}_k(\bar{e}) = \bigcup_{\substack{\bar{e}'\in \mathcal{E}^+_{0,k}\\ r(\bar{e}') = r(\bar{e}|_k)}}f_{\bar{e}|_k}^{-1}\circ f_{\bar{e}'}(t_{s(\bar{e}')})
        \end{equation}
      viewed as a tiled patch, where the tiles are the sets $f_{\bar{e}|_k}^{-1}\circ f_{\bar{e}'}(t_{s(\bar{e}')})$ for a path $\bar{e}'\in\mathcal{E}^+_{0,k}$ with $r(\bar{e}') = r(\bar{e}|_k)$. The hypotheses on the maps $f_e$ guarantee that the approximants are nested, i.e., we have the inclusion of patches
      $$\{0\}\subset t_{s(\bar{e})}\subset \mathcal{P}_1(\bar{e})\subset \cdots \subset \mathcal{P}_k(\bar{e})\subset \mathcal{P}_{k+1}(\bar{e})\subset \cdots.$$
      Patches of the form $\mathcal{P}_k(\bar{e})$ are called \textbf{level $k$-supertiles}.
      \begin{definition}
        \label{def:tiling}
        For $\bar{e}\in X^+_{\mathcal{B}_x}$, the tiling $\mathcal{T}_{\bar{e}}$ is the largest tiled subset of $\mathbb{R}^d$ such that $\mathcal{P}_k(\bar{e})$ is a patch of $\mathcal{T}_{\bar{e}}$ for all $k$ and each tile of $\mathcal{T}_{\bar{e}}$ is contained in all but finitely many of the approximants $\mathcal{P}_k(\bar{e})$. In other words,
        $$\mathcal{T}_{\bar{e}} = \bigcup_{k>0} \mathcal{P}_k(\bar{e}).$$
      \end{definition}
      Some care needs to be given in order to produce tilings which 1) cover all of $\mathbb{R}^d$ and 2) have finite local complexity, as nothing guarantees that the tiling in $\mathcal{T}_{\bar{e}}$ to have either property. The first property needed is the following.
      \begin{definition}
        A collection $\mathcal{F} = \{\mathcal{F}_1,\dots, \mathcal{F}_N\}$ of substitution rules is \textbf{uniformly expanding} if there exist numbers $\theta_1,\dots, \theta_N\in(0,1)$ such that each substitution rule $f_{i,j,k}\in\mathcal{F}_\ell$ is of the form $f_{i,j,k}(x) = \theta_\ell x + \tau_{i,j,k}$ for some $\tau_{i,j,k}\in\mathbb{R}^d$ .
      \end{definition}
      \begin{definition}
        A collection $\mathcal{F} = \{\mathcal{F}_1,\dots, \mathcal{F}_N\}$ of substitution rules is \textbf{compatible} if, for any $x\in\Sigma_N$ and $\bar{e}^+\in X^+_{\mathcal{B}_x}$ such that $\mathcal{T}_{\bar{e}}$ defined in Definition \ref{def:tiling} covers all of $\mathbb{R}^d$, then $\mathcal{T}_{\bar{e}}$ has finite local complexity.
      \end{definition}
Compatibility is automatic in for $d=1$. The results of \cite{GKM:computer} show that this is not asking for too much in higher dimensions.      The following are standard, see \cite[Lemma 5]{ST:random}.

      \begin{lemma}
  \label{lemma:deform}
        Let $\mathcal{F} = \{\mathcal{F}_1,\dots,\mathcal{F}_N\}$ be a collection of compatible and uniformly expanding substitution rules defined on the same set of prototiles. For $x\in X_\mathcal{F}$ consider the Bratteli diagram $\mathcal{B}_x$ where the edge set $\mathcal{E}_k$ is defined by $\mathcal{F}_{x_k}$. Then:
  \begin{enumerate}
  \item If $\bar{e}\sim \bar{e}^*\in X^+_{\mathcal{B}_x}$ then there exists $\tau\in \mathbb{R}^d$ such that $\mathcal{T}_{\bar{e}^*} = \mathcal{T}_{\bar{e}} + \tau$.
  \item $\Omega_{\mathcal{T}_{\bar{e}}}$ only depends on the minimal component: $\Omega_{\mathcal{T}_{\bar{e}}} = \Omega_{\mathcal{T}_{\bar{e}'}}$ for all $\bar{e}'\in \overline{[\bar{e}]}$.
\end{enumerate}
\end{lemma}
Let $\mathring{X}^+_\mathcal{B}\subset X^+_{\mathcal{B}}$ be the set of paths $\bar{e}$ such that $\mathcal{T}_{\bar{e}}$ covers all of $\mathbb{R}^d$. Note that by the previous Lemma, if $\mathcal{B}^+_x$ is minimal, then $\Omega_{\mathcal{T}_{\bar{e}_1}} = \Omega_{\mathcal{T}_{\bar{e}_2}}$ for any $\bar{e}_1,\bar{e}_2\in X_{\mathcal{B}}^+$. In such cases we denote the tiling space simply by $\Omega_{\mathcal{B}}$ or, if $\mathcal{B}$ is defined by a parameter $x\in X_\mathcal{F}$, we write $\Omega_x$. The following is a consequence of the previous Lemma. 
\begin{corollary}
  \label{cor:Robinson}
Let  $\mathcal{F}$ a family of $N$ uniformly expanding and compatible substitutions, $x\in X_\mathcal{F}$, and $\mathcal{B}_x$ be a Bratteli diagram such that the set $\mathcal{E}_k^+$ in $\mathcal{B}^+_x$ is defined by $\mathcal{F}_{x_k}$. Suppose $\mathcal{B}_x^+$ is minimal. Then the assignment $\bar{e}\mapsto\mathcal{T}_{\bar{e}}$ defines a surjective, continuous map $\bar{\Delta}_x:\mathring{X}^+_{\mathcal{B}_x}\rightarrow \mho_{x}$, where $\mho_x$ is the canonical transversal of $\Omega_{\mathcal{T}_{\bar{e}}}$, $\bar{e}\in \mathring{X}^+_{\mathcal{B}}$. 
\end{corollary}
\begin{definition}
  \label{def:minimal}
  A probability measure $\mu$ on $\Sigma_N$ is \textbf{minimal} if the set of $x$ for which $\mathcal{B}_x$ is minimal has full measure.
\end{definition}
The following is \cite[Proposition 2]{ST:random}.
\begin{proposition}
  \label{prop:measBij}
Let $\mathcal{F}$ a family of $N$ uniformly expanding and compatible substitutions, $x\in X_\mathcal{F}$, $\mathcal{B}_x^+$ a minimal Bratteli diagram such that the set $\mathcal{E}_k^+$ in $\mathcal{B}^+_x$ is defined by $\mathcal{F}_{x_k}$. Suppose that $\mu(\mathring{X}^+_{\mathcal{B}}) = 1$ for any probability measure $\mu$ on $X^+_{\mathcal{B}}$ which is invariant under the tail equivalence relation. Then the map $\bar{\Delta}_x$ in Corollary \ref{cor:Robinson} provides a bijection between measures $\mu$ on $X^+_\mathcal{B}$ which are invariant under the tail equivalence relation and measures on $\mho_x$ which are holonomy-invariant.
\end{proposition}
\begin{proposition}
  \label{prop:UE}
  Let  $\mathcal{F}$ a family of $N$ uniformly expanding and compatible substitutions, and $\mu$ a minimal, ergodic $\sigma$-invariant Borel probability measure on $X_\mathcal{F}$. Then for $\mu$-almost every $x\in X_\mathcal{F}$ we have that there is a unique probability measure $\mu_x$ on $X^+_{\mathcal{B}_x}$ which is invariant under the tail equivalent relation. Moreover, we have that $\mu_x(\mathring{X}^+_{\mathcal{B}_x}) = 1$ and there is a unique $\mathbb{R}^d$-invariant probability measure on $\Omega_x$.
\end{proposition}
\begin{proof}
  For $x\in X_\mathcal{F}$, define $\lambda: X_\mathcal{F}\rightarrow \mathbb{R}$ by
  $$\lambda_x = \limsup_{k\rightarrow \infty}\frac{\log |\mathcal{E}^x_{0,k}|}{k},$$
for all $x\in X_\mathcal{F}$, where $|\mathcal{E}^x_{0,k}|$ is the number of paths from $\mathcal{V}_0$ to $\mathcal{V}_k$ on $\mathcal{B}_x$. Note that this is a $\sigma$-invariant function, so it is constant $\mu$-almost everywhere. Denote by $\lambda_\mu$ this value and $A_\mu\subset X_\mathcal{F}$ the full $\mu$-measure set such that $\lambda_x = \lambda_\mu$ for all $x\in A_\mu$.
  
  Let $x\in A_\mu \cap \mathrm{supp}\, \mu$ be a Poincar\'{e} recurrent point and let $\mathcal{B}_x$ be the corresponding Bratteli diagram. By minimality there exists a $k^*>0$ such that for any $v\in\mathcal{V}^+_0$ and $w\in\mathcal{V}^+_{k^*}$ there is a path $\bar{p}\in \mathcal{E}^+_{0,k^*}$ with $s(\bar{p}) = v$ and $r(\bar{p}) = w$. Let $U_x\subset X_\mathcal{F}$ be  the cylinder set defined by $U_x = \{y\in\bar{\Sigma}_N:y_i = x_i\mbox{ for all }i=1,\dots, k^* \}$ and note that $\mu(U_x)>0$. Let $k_i\rightarrow \infty$ be the sequence of first return times to $U_x$ for $x$. That is, $\sigma^{k_i}(x)\in U_x$ for all $i>0$ and $\sigma^{k}(x)\not\in U_x$ if $k\neq k_i$ for some $i$. Note that by the definitions of $U_x$, $k_i$ and $k^*$ there is a positive matrix $M_x$ such that the number of paths between $v\in \mathcal{V}^+_{k_i}$ and $w\in \mathcal{V}^+_{k_i+k^*}$ is given by $M_x(v,w)$. Let $\lambda_{PF}$ denote the Perron-Frobenius eigenvalue of $M_x$. Then for any $\varepsilon$ there exists a $C_\varepsilon$ such that for $v\in \mathcal{V}^+_{k'}$ with $k_i\leq k'< k_{i+1}$ we have that
\begin{equation}
  \label{eqn:PFestimate}
  |\mathcal{E}^+_{0,k_i}|\geq C_\varepsilon\left( \frac{\lambda_x}{e^\varepsilon}\right)^i.
\end{equation}
Since $\mu(U_x)>0$ it follows from the estimate above that $\lambda_\mu \geq \log\lambda_{PF} - \varepsilon$ for any $\varepsilon>0$, so $\lambda_\mu>0$. So for any $\varepsilon>0$ there is a $C_\varepsilon$ so that $|\mathcal{E}^+_{0,k}|\geq C_\varepsilon e^{(\lambda_\mu - \varepsilon)k}$ for all $k>0$. It follows from this, minimality and recurrence of $x$ that for the two quantities
$$\lambda^-:= \min_{v\in \{v_1,\dots, v_M\}} \left\{ \liminf_{k\rightarrow \infty}\frac{\log |\mathcal{E}^+_{v}|}{k}\right\} \hspace{.4in}\mbox { and } \hspace{.4in} \lambda^+:= \max_{v\in \{v_1,\dots, v_M\}} \left\{ \limsup_{k\rightarrow \infty}\frac{\log |\mathcal{E}^+_{v}|}{k}\right\},$$
we have that $\lambda^- = \lambda^+\geq \lambda_{PF}>0$. That $\mu(\mathring{X}^+_{\mathcal{B}_x}) = 1$ now follows by \cite[Lemma 3]{ST:random} for any Borel probability measure $\mu$ which is invariant under the tail-equivalence relation. That there is a unique such measure follows from the main result of \cite{trevino:bratMasur}, so the uniqueness of an invariant measure on $\Omega_x$ follows from Proposition \ref{prop:measBij}.
\end{proof}
\subsection{Renormalization}
\label{subsec:renorm}
There are two advantages of using bi-infinite Bratteli diagrams as opposed to the usual diagrams indexed by $\mathbb{N}$. The first is that the path space of a bi-infinite diagram $\mathcal{B}_x$ parametrizes all tilings in a tiling space $\Omega_x$ in a continuous way (see Proposition \ref{prop:Robinson2} in \S \ref{subsec:renorm}), and not just the ones associated to canonical transversals, as it happens with traditional (one-sided) Bratteli diagrams. This permits one to transfer properties back and forth between the path space of the bi-infinite diagram and the corresponding tiling space (see Proposition \ref{prop:measBij}). The second and more important advantage is that one can shift the labels of a diagram $\mathcal{B}_x$ to obtain a diagram $\mathcal{B}_{\sigma(x)}$ and this process is equivariant with a homeomorphism of tiling spaces $\Phi_x:\Omega_{x}\rightarrow \Omega_{\sigma(x)}$. Having $x$ belong to a two-sided shift allows this operation to be invertible, which will allow for the semi-invertible Oseledets theorem to be applied in \S \ref{subsec:TrCocycle}.

Consider a minimal measure $\mu$ and note that being minimal is a $\sigma$-invariant property: $\mathcal{B}_x$ is minimal if and only if $\mathcal{B}_{\sigma(x)}$ is. As such, for an $\sigma$-invariant ergodic Borel probability measure $\mu$ then the set of minimal diagrams $\mathcal{B}_x$ has either full or null measure.


Let
$$\mathring{X}_\mathcal{B} = \{(x^-,x^+)\in X_\mathcal{B}: x^+\in \mathring{X}^+_\mathcal{B}\}.$$
\begin{proposition}
  \label{prop:Robinson2}
  Let $\mathcal{F}$ be a family of $N$ uniformly expanding and compatible substitutions on a set of prototiles, and suppose that $\mathcal{B}_x$ is minimal. Then the map $\bar{\Delta}_x$ from Corollary \ref{cor:Robinson} extends to a continuous surjective map $\Delta_x:\mathring{X}_\mathcal{B}\rightarrow \Omega_x$
  \end{proposition}
\begin{proof}
  Let $\bar{e} = (e^-,e^+)\in\mathring{X}_{\mathcal{B}_x}$. The discussion leading to Corollary \ref{cor:Robinson} shows how $e^+\in \mathring{X}^+_{\mathcal{B}_x}$ determines a point in the canonical transversal $\mho_x\subset \Omega_x = \Omega_{\mathcal{T}_{e^+}}$. It is left to show what role $e^-$ plays.

  What $e^-$ determines is a vector $\tau_{e^-}$ so that $\Delta(\bar{e}) = \varphi_{\tau_{e^-}}(\mathcal{T}_{e^+})$ and this is done as follows (there is a concrete example worked out in section \ref{sec:interlude}, in case the reader would find that helpful as they read the construction). Consider the tile $t$ containing the origin in $\mathcal{T}_{e^+}$. The assumptions about the subtitution rules imply that the origin is in the interior of this tile, and it can be subdivided according to the substitution rule $\mathcal{F}_{x_{-1}}$ into $|r^{-1}(v_t)|\geq 1$ tiles, where $v_t\in\mathcal{V}_0$ is the vertex identified with the tile $t$ containing the origin. The edge $e_{-1}$ corresponds to a choice of one of the smaller tiles which make up $t$. Now, $\mathcal{F}_{x_{-2}}$ gives a rule for subdividing this tile into $|r^{-1}(s(e_{-1}))|\geq 1$ smaller tiles and the edge $e_{-2}$ corresponds to choosing one of the smaller tiles in this subdivision. Carrying on recursively, after ending up with a small connected subset at level $-k$, the substitution rule $\mathcal{F}_{x_{-k-1}}$ yields a collection of smaller pieces which make up this connected subset and the edge $e_{-k-1}$ of $\bar{e}$ determines a choice of one of the smaller pieces. Since $S^-(\bar{e})\leq c$, on average, the pieces are contracting at a rate of $e^{-ck}$. Thus performing this procedure infinitely many times yields a unique point $p_{e^-}\in t$. The vector $\tau_{e^-}$ is now defined to be the unique vector which takes $p_{e^-}\in t \in \mathcal{T}_{e^+}$ to the origin. That is, the point $\varphi_{\tau_{e^-}}(p_{e}) = 0$. This assignment can readily be seen to be continuous.
\end{proof}
Let $\mathcal{B}_x$ be a Bratteli diagram determined by a family of substitution rules $\mathcal{F}_1,\dots, \mathcal{F}_N$ and a point $x\in X_\mathcal{F}$. There is a natural homeomorphism $h_x:X_{\mathcal{B}_x}\rightarrow X_{\mathcal{B}_{\sigma(x)}}$ defined by the shifting of indices in $X_{\mathcal{B}_x}$ by 1. This yields a homeomorphism of tiling spaces, which is proved in \cite[Proposition 6]{ST:random}.
\begin{proposition}
  \label{prop:renorm}
Let $\mathcal{F} = \{\mathcal{F}_1,\dots,\mathcal{F}_N\}$ be a family of uniformly expanding and compatible substitution rules and suppose that $\mathcal{B}_x$ is minimal. The shift $\sigma:X_\mathcal{F}\rightarrow X_\mathcal{F}$ induces a homeomorphism of tiling spaces $\Phi_x:\Omega_x\rightarrow \Omega_{\sigma(x)}$ satisfying $\Phi_x\circ \Delta_x = \Delta_{\sigma(x)}\circ h_x$. In addition, level-$k$ supertiles on $\mathcal{T}_{\bar{e}}\in \Omega_x$ are mapped to level-$k-1$ supertiles on $\Phi_x(\mathcal{T}_{\bar{e}}) = \mathcal{T}_{\sigma(\bar{e})}\in\Omega_{\sigma(x)}$.
\end{proposition}
\section{Interlude: an example}
\label{sec:interlude}
Before proceeding to the second, more technical part of the paper, I will take the time to relate the example in Figure \ref{fig:HalfHexes} to the constructions of tilings from Bratteli diagrams in \S \ref{sec:brat}, and to the renormalization procedure in \S \ref{subsec:renorm}.

Figure \ref{fig:HalfHexes} illustrates part of two different substitution rules on six different prototiles, which are rotated copies of the half-hexagon prototile illustrated in Figure \ref{fig:HalfHexes} by $2\pi k/6$, $k = 1,\dots, 5$. The substitution rules for the rest of the prototiles in these cases are then defined by looking at the substitution for the prototile in Figure \ref{fig:HalfHexes} and rotating them by $2\pi k/6$, $k=1,\dots, 6$. The graphs for the corresponding graph iterated function systems which define the substitution rules as in (\ref{eqn:GIFS}) are illustrated in Figure \ref{fig:HalfHexesGIFS}.
      \begin{figure}[t]
        \centering
        \includegraphics[width = 6.25in]{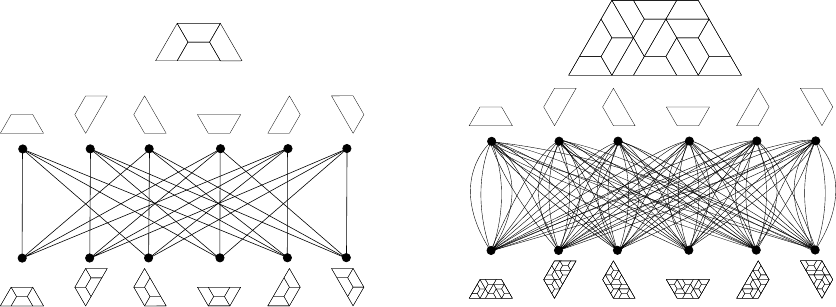}
        \caption{The associated graph iterated functions systems associated with the substitution rule in Figure \ref{fig:HalfHexes}.}
        \label{fig:HalfHexesGIFS}
      \end{figure}

      To connect this more concretely with the substitution rule expressed in (\ref{eqn:GIFS}), we take each prototile as a subset of $\mathbb{R}^2$ with its barycenter at the origin (note that this choice will define the canonical transversal). Each edge in the graphs of Figure \ref{fig:HalfHexesGIFS} corresponds to a contracting linear map from (\ref{eqn:GIFS}) which places a scaled copy of a prototile inside another prototile, and so we can express the subset of $\mathbb{R}^2$ corresponding to a prototile as the union of images of contracting linear maps, i.e., as in (\ref{eqn:GIFS}).

\begin{wrapfigure}{l}{.45\linewidth}
  \centering
  \includegraphics[width = 3in]{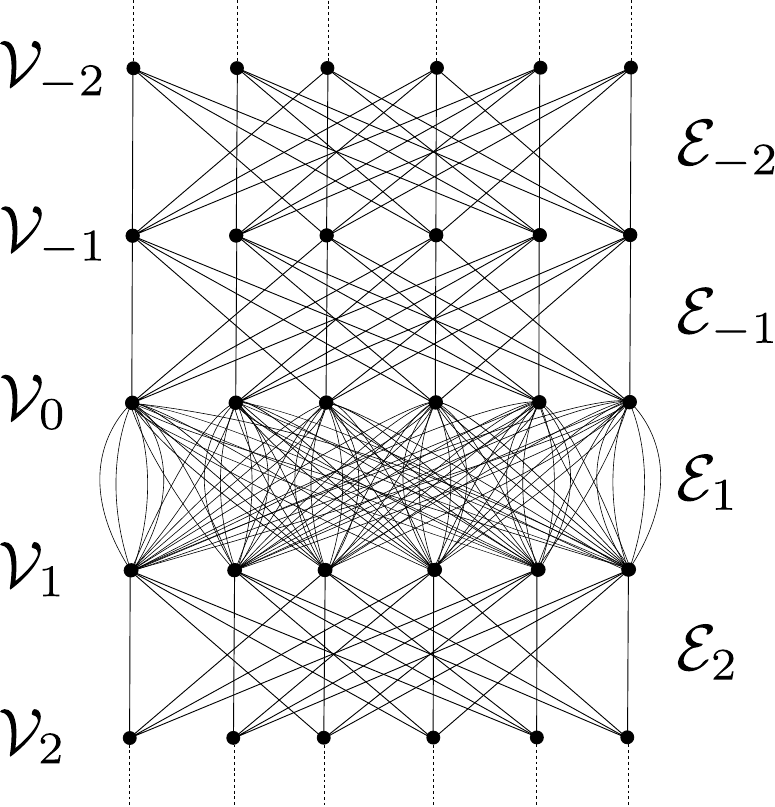}
  \caption{The Bratteli diagram $\mathcal{B}_x$ for any $x\in C([00.10])$ looks like this around $\mathcal{V}_0$.}
  \label{fig:HalfHexesBrat}
\end{wrapfigure}
      
Given that these two substitution rules are primitive, for any point $x\in\Sigma_2$, we obtain a minimal Bratteli diagram $\mathcal{B}_x$ where the edge information $\mathcal{E}_k$ is given by the graph on the left in Figure \ref{fig:HalfHexesGIFS} if $x_k = 0$, and otherwise by the graph on the right. Consider now a point $x\in \Sigma_2$ where $x = (\dots, x_{-2},x_{-1}.x_1,x_2,\dots) = (\dots, 0,0.1,0,\dots) \in C([00.10])$, where the dot $.$ denotes the break between the negative and positive parts, and $C([w])\subset \Sigma_2$ denotes the obvious cylinder set in $\Sigma_2$ defined by the word $w$ defined for a specific set of indices. Figure \ref{fig:HalfHexesBrat} illustrates the common part of a Bratteli diagram $\mathcal{B}_x$ for any $x\in C([00.10])$.

Consider now the positive part $\mathcal{B}_x^+$ of the Bratteli diagram $\mathcal{B}_x$ for $x\in C([00.10])$, and its associated path space $X_{\mathcal{B}_x}^+$, and consider a path $\bar{e}^+ = (e_1,e_2,\dots)\in X_{\mathcal{B}_x}^+$, the first two edges of which are outlined in bold blue in the left part of Figure \ref{fig:path}. This path defines both a cylinder set $C((e_1,e_2))\subset X_{\mathcal{B}_x}^+$ as in (\ref{eqn:cylinderBrat}), as well as a second approximant $\mathcal{P}_2(\bar{e})$ as in (\ref{eqn:approx}), which is denoted on the right part of Figure \ref{fig:path}. As such, it also denoted cylinder sets $\mathcal{C}_{\mathcal{P}_2(\bar{e}^+),b}$, where $b$ is the distinguished point corresponding to the barrycenter of the tile. In fact, using the continuous map $\bar{\Delta}_x$ from Corollary \ref{cor:Robinson}, it follows that $\bar{\Delta}_x\left(C((e_1,e_2))\right) = \mathcal{C}_{\mathcal{P}_2(\bar{e}^+),b}$.

\begin{figure}[h]
  \centering
  \includegraphics[width = 5.5in]{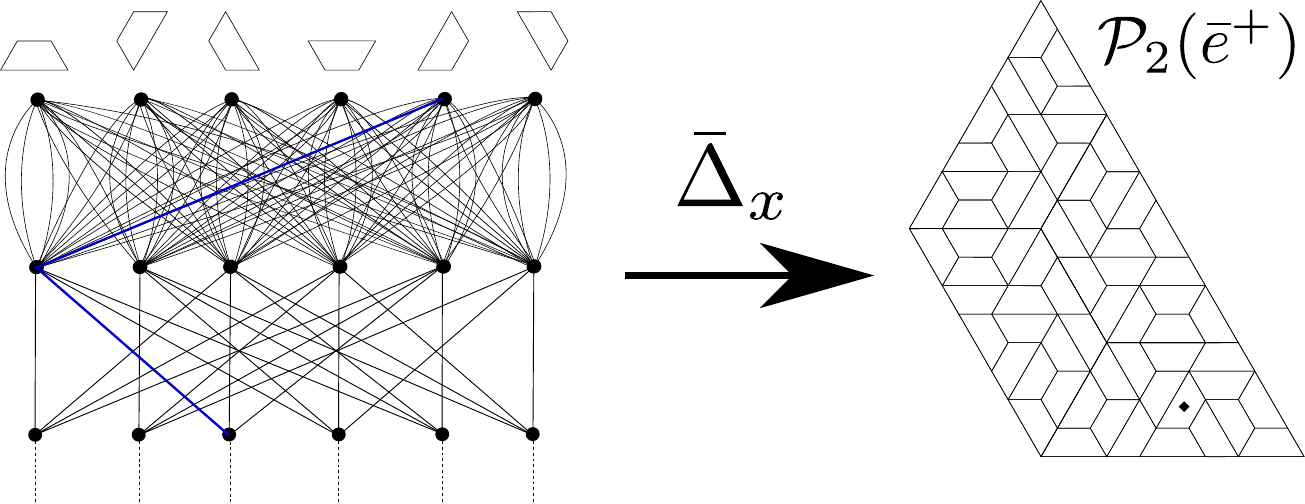}
  \caption{Mapping a cylinder set of the positive part to a cylinder set on the canonical transversal.}
  \label{fig:path}
\end{figure}

Now onto the negative part $\mathcal{B}_x^-$. A finite path $\bar{e}^- = (e_{-\ell},\dots, e_{-1})$ on $\mathcal{B}_x^-$ defines a cylinder set $C(\bar{e}^-)\subset X_{\mathcal{B}_x}^-$, as well as a measurable subset of the prototile associated with the vertex $r(\bar{e}^-)\in\mathcal{V}_0$. This subset is precisely $f_{e_{-1}}\circ \cdots \circ f_{e_{-\ell}}(t_{s(\bar{e}^-)})\subset t_{r(\bar{e}^-)}$, where $f_{-i}$ is the contracting map associated with the edge $e_{-i}\in\mathcal{E}_{-i}$, and $t_{r(e)}$ is the prototile corresponding to the vertex $r(e)\in\mathcal{V}_0$. As such, the blue path $\bar{e}^- = (e_{-1},e_{-2})$ denoted in bold blue in Figure \ref{fig:negative} defines a cylinder set $C(\bar{e}^-)\subset X_{\mathcal{B}_x}^-$ and, on the right, the associated subset denoted in blue on the tile.

\begin{figure}[h]
  \centering
  \includegraphics[width = 4.5in]{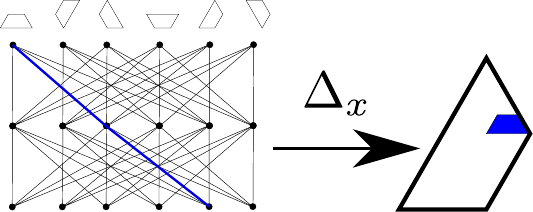}
  \caption{Mapping a cylinder set of the negative part to a ``cylinder set'' of a prototile.}
  \label{fig:negative}
\end{figure}

Putting Figures \ref{fig:path} and \ref{fig:negative} together one obtains a path $\bar{e}' = (e_{-2},e_{-1},e_1,e_2)$ which defines a cylinder set $C(\bar{e}')\subset X_{\mathcal{B}_x}\subset X_{\mathcal{B}_x}^-\times X_{\mathcal{B}_x}^+$. The image of this cylinder set under the map $\Delta_x:\mathring{X}_{\mathcal{B}_x}\rightarrow \Omega_x$ from Proposition \ref{prop:Robinson2} is a cylinder set in $\Omega_x$, althought not of the canonical form as in (\ref{eqn:openSet}). In any case, the cylinder set $\Delta_x(C(\bar{e}'))$ is described as all the tilings in $\Omega_x$ having a patch around the origin which is a translation copy of $\mathcal{P}_2(\bar{e}^+)$ in Figure \ref{fig:path}, and where the origin is somewhere in the blue region of Figure \ref{fig:negative}. This is illustrated in Figure \ref{fig:full}.

\begin{figure}[h]
  \centering
  \includegraphics[width = 6in]{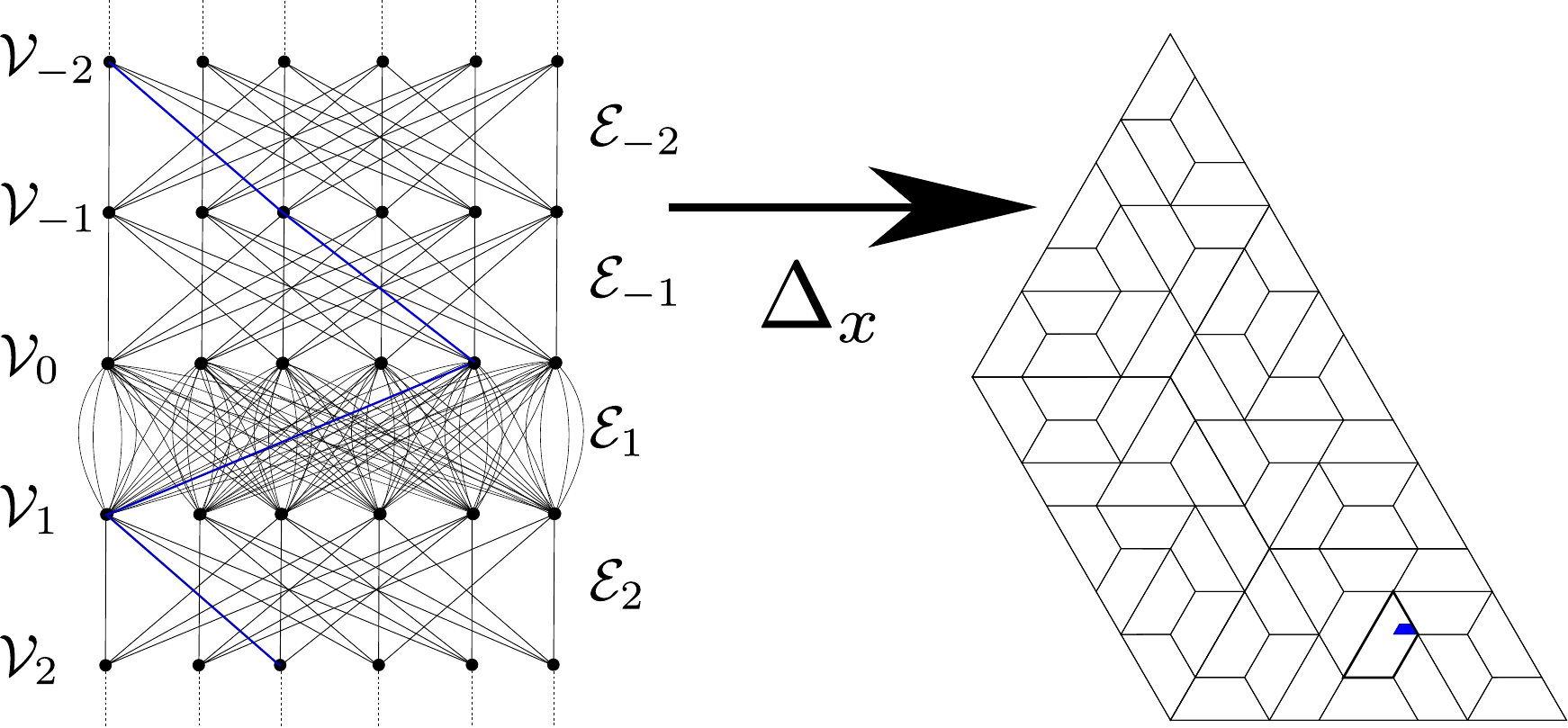}
  \caption{Mapping a cylinder set to a cylinder set.}
  \label{fig:full}
\end{figure}

It remains to illustrate how renormalization works in this example. As described in \S \ref{subsec:renorm}, renormalization is driven by the shift $\sigma:\Sigma_2\rightarrow \Sigma_2$. Figure \ref{fig:equivariance} illustrates what a step of renormalization does to the cylinder set in Figure \ref{fig:full}. The illustration uses the inverse of the shift, as it shows the relationship between renormalization and the substitutions encoded in $\mathcal{B}_x$.

\begin{figure}[h]
  \centering
  \includegraphics[width = 5.5in]{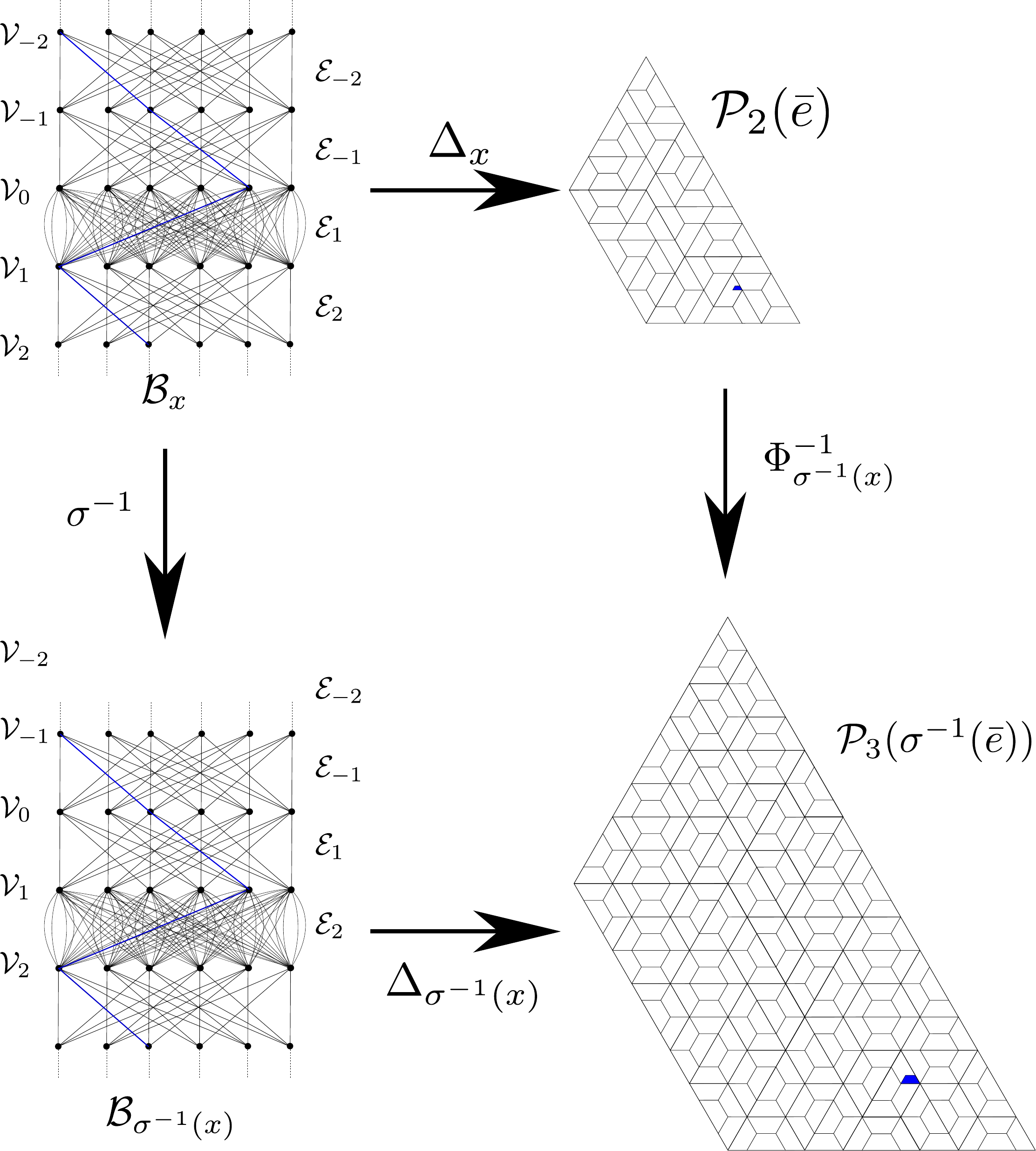}
  \caption{The mechanism of renormalization: the process of applying the inverse $\sigma^{-1}$ of the shift corresponds to applying the substitution and expansion rule defined by the edge set $\mathcal{E}_{-1}$ on $\mathcal{B}_x$. This shifts levels on $\mathcal{B}_x$ to obtain $\mathcal{B}_{\sigma^{-1}(x)}$ and maps level-$k$ supertiles to level-$(k+1)$, as shown with the second approximant supertile from Figure \ref{fig:path}. At the level of cylinder sets defined by the finite path in bold from Figure \ref{fig:full}, we have that $\Phi_{\sigma^{-1}}^{-1}(\Delta_x(C(\bar{e}))) = \Delta_{\sigma^{-1}(x)}(C(\sigma^{-1}(\bar{e})))$.}
  \label{fig:equivariance}
\end{figure}
\section{LF Algebras and traces}
\label{sec:LF}
A \textbf{multimatrix algebra} is a $*$-algebra of the form
$$\mathcal{M} = M_{\ell_1}\oplus \cdots \oplus M_{\ell_n},$$
where $M_\ell$ denotes the algebra of $\ell\times\ell$ matrices over $\mathbb{C}$. Let $\mathcal{M}_1 = M_{\ell_{1,1}}\oplus \cdots \oplus M_{\ell_{n,1}}$ and $\mathcal{M}_2 = M_{\ell_{1,2}}\oplus \cdots \oplus M_{\ell_{n,2}}$ be multi-matrix algebras and suppose $\phi:\mathcal{M}_1\rightarrow \mathcal{M}_2$ is a unital homomorphism of $\mathcal{M}_1$ into $\mathcal{M}_2$. Then $\phi$ is determined up to unitary equivalence in $\mathcal{M}_2$ by a $\ell_{n,2}\times \ell_{n,1}$ non-negative integer matrix $A_\phi$ \cite[\S III.2]{davidson:book}. It follows that the inclusion of a multi-matrix algebra $\mathcal{M}_0$ into a larger multimatrix algebra $\mathcal{M}_1$ is determined up to unitary equivalence by a matrix $A_0$ which roughly states how many copies of a particular subalgebra of $\mathcal{M}_0$ goes into a particular subalgebra of $\mathcal{M}_1$.

Let $\mathcal{B}$ be a Bratteli diagram and let $A_k^+$, $k\in\mathbb{N}$, be the connectivity matrix at level $k$. In other words, $A_k(i,j)^+$ is the number of edges going from $v_j\in \mathcal{V}_{k-1}$ to $v_i\in\mathcal{V}_k$. An analogous matrix $A^-_k$ can be defined for $k<0$. Starting with $\mathcal{M}_0 = \mathbb{C}^{|\mathcal{V}_0|}$ the matrices $A_k^\pm$ define two families of inclusions $i_{|k|}^\pm:\mathcal{M}_{|k|-1}^\pm\rightarrow \mathcal{M}_{|k|}^\pm$ (up to unitary equivalence), one for $+$ and one for $-$, where each $\mathcal{M}_k^\pm$ is a multimatrix algebra. More explicitly, if
$$\mathcal{M}_k^\pm = M_{n_1^\pm}\oplus \cdots \oplus M_{n_k^\pm}$$
then starting with the vector $h^0=(1,\dots, 1)^T\in \mathbb{C}^{|\mathcal{V}_0|}$ and defining $h^{k,+} = A_k^+h^{k-1,+} = A_k^+\cdots A_{1}^+ (h^0)^T$ for $k\geq 0$ and $h^{k,-} =  h^{k-1,-}A_k^- =  h^0 A_1^-\cdots A_{k}^- $ for $k\leq 0$, we have that
$$\mathcal{M}_k^+ = M_{h_1^{k,+}}\oplus \cdots \oplus M_{h_{n_k}^{k,+}}\hspace{.5in}\mbox{ and }\hspace{.5in}\mathcal{M}_k^- = M_{h_1^{k,-}}\oplus \cdots \oplus M_{h_{n_k}^{k,-}}$$
and the inclusions $i_{|k|}^\pm:\mathcal{M}_{|k|-1}^\pm\rightarrow \mathcal{M}_{|k|}^\pm$ are defined up to unitary equivalence by the matrices $A_k^\pm$. With these systems of inclusions one can define the inductive limits
\begin{equation}
 LF(\mathcal{B}^+) := \bigcup_k\mathcal{M}_k^+ = \lim_{\rightarrow}(\mathcal{M}_k^+,i_k^+) \hspace{.7in}  LF(\mathcal{B}^-) := \bigcup_k\mathcal{M}_k^- = \lim_{\rightarrow}(\mathcal{M}_k^-,i_k^-)
\end{equation}
which are $*$-algebras called the \textbf{locally finite (LF) algebras} defined by $\mathcal{B}$. Their $C^*$-completion
$$AF(\mathcal{B}^+) := \overline{LF(\mathcal{B}^+)},\hspace{.8in AF(\mathcal{B}^-) := \overline{LF(\mathcal{B}^-)}}$$
are the \textbf{approximately finite-dimensional (AF) algebras} defined by $\mathcal{B}$.
\begin{definition}
  A \textbf{trace} on a $*$-algebra $\mathcal{A}$ is a linear functional $\tau:\mathcal{A}\rightarrow \mathbb{C}$ which satisfies $\tau(ab) = \tau(ba)$ for all $a,b\in\mathcal{A}$\footnote{There is no assumption that traces are positive (that is, $\tau(aa^*)>0$).}. The set of all traces of $\mathcal{A}$ forms a vector space over $\mathbb{C}$ and it is denoted by $\tr(\mathcal{A})$. A \textbf{cotrace} $\tau'$ is an element of the dual vector space $\tr^*(\mathcal{A}):=\tr(\mathcal{A})^*$.
  \end{definition}
For $M_\ell$, the algebra of $\ell\times\ell$ matrices, $\tr(M_\ell)$ is one-dimensional and generated by the trace $\tau_\ell: a\mapsto \sum_{i=1}^\ell a_{ii}$. For a multimatrix algebra $\mathcal{M} = M_{\ell_1}\oplus\cdots\oplus M_{\ell_n}$, the dimension of $\tr(\mathcal{M})$ is $n$ and is generated by the traces $\tau_{\ell_i}\in \tr(M_{\ell_i})$ for $i=1,\dots, n$.

Let $i_k^+:\mathcal{M}_{k-1}^+\rightarrow \mathcal{M}_k^+$ be the family of inclusions defined by the positive part of a Bratteli diagram $\mathcal{B}^+$. Then there is a dual family of inclusions $i_k^*:\tr(\mathcal{M}_k^+)\rightarrow \tr(\mathcal{M}_{k-1}^+)$ (and an analogous family $i_k^*:\tr(\mathcal{M}_k^-)\rightarrow \tr(\mathcal{M}_{k-1}^-)$). The trace spaces of the LF algebras defined by a Bratteli diagram $\mathcal{B}$ are then the inverse limits
\begin{equation}
  \begin{split}
    \tr(\mathcal{B}^+) &:= \tr(LF(\mathcal{B}^+)) = \lim_{\leftarrow} (i^*_k,\tr(\mathcal{M}^+_k)) \\
    \tr(\mathcal{B}^-) &:= \tr(LF(\mathcal{B}^-)) = \lim_{\leftarrow} (i^*_k,\tr(\mathcal{M}^-_k))
    \end{split}
  \end{equation}
which are vector spaces. The respective spaces of cotraces are then
$$\tr^*(\mathcal{B}^+)  = \lim_{\rightarrow} ((i^*_k)^*,\tr^*(\mathcal{M}^+_k))\hspace{.4in}\mbox{ and } \hspace{.4in} \tr^*(\mathcal{B}^-)  = \lim_{\rightarrow} ((i^*_k)^*,\tr^*(\mathcal{M}^-_k)).$$
\begin{remark}
Note that since every class $[p]$ of the dimension group $K_0(AF(\mathcal{B}^+))$ can be represented by an element $p\in LF(\mathcal{B}^+)$, the set $\tr(\mathcal{B}^+)$ also defines the dual space $\tr(K_0(AF(\mathcal{B}^+))) := K_0(AF(\mathcal{B}^+))'$. As such, the trace spaces which will be used can be thought of as the dual of the invariant $K_0(AF(\mathcal{B}^+))$.
\end{remark}

Let $\{\mathcal{B}_x\}$ be a family of Bratteli diagrams parametrized by $x\in X\subset \Sigma_N$, where $X$ is a closed, $\sigma$-invariant subset of $\Sigma_N$ (an example of this is $X_\mathcal{F}$, where $\mathcal{F}$ is a family of substitutions on $N$ tiles, as described in \S \ref{subsec:renorm}).  In what follows, we will focus on the invariants defined by the positive part of $\mathcal{B}_x$, so we will drop the $+$ superscripts used earlier. The shift induces a $*$-homomorphism $\sigma_*:\mathcal{M}_0^x\rightarrow \mathcal{M}_0^{\sigma(x)}$ as follows. For $a = (a_1,\dots a_{|\mathcal{V}_0^+|})\in\mathcal{M}_0^x$ consider its image $i^x_1a = ((i^x_1a)_1,\dots, (i^x_1a)_{|\mathcal{V}_1^+|})\in \mathcal{M}_1^x$. Composing this with the evaluation by $\mathcal{T}_1^x $ which takes $a = (a_1,\dots, a_{|\mathcal{V}_1|})\in \mathcal{M}_1^x$ to $\mathcal{T}_1^x(a) = (\tau_1(a_1), \dots, \tau_{|\mathcal{V}_1^+|}(a_{|\mathcal{V}_1^+|}))\in\mathcal{M}_0^{\sigma(x)}$, we obtain the map
\begin{equation*}
  \label{eqn:traces!}
\begin{split}
  \sigma_* = \mathcal{T}_1^x\circ i^x_1 : a\mapsto \left(\tau_1(i^x_1a),\dots, \tau_{|\mathcal{V}_1^+|}(i^x_1a) \right) &= \left( \sum_j a_j A_1(1,j), \dots, \sum_ja_j A_1(|\mathcal{V}_1|,j)\right) \\
   &= A_1(a_1,\dots, a_{|\mathcal{V}_0|})^T \in \mathcal{M}_0^{\sigma(x)} = \mathbb{C}^{|\mathcal{V}_1|}.
\end{split}
\end{equation*}
  As such, the map $\sigma_*: \mathcal{M}_0^{x} = \mathbb{C}^{|\mathcal{V}_0|} \rightarrow \mathbb{C}^{|\mathcal{V}_1|} = \mathcal{M}_0^{\sigma(x)} $ coincides with the linear map $A_1:\mathbb{C}^{|\mathcal{V}_0|}\rightarrow \mathbb{C}^{|\mathcal{V}_1|}$ defined by the first matrix of the Bratteli diagram. As such there is a dual map $\sigma^*:\tr(\mathcal{M}_0^{\sigma(x)})\rightarrow \tr(\mathcal{M}_0^{x})$ and so we have the isomorphisms
  \begin{equation}
    \label{eqn:traceSpaces}
\tr(\mathcal{B}^+_x) \cong \lim_{\leftarrow}\left( \tr\left(\mathcal{M}_0^{\sigma^k(x)}\right),\sigma^* \right)\hspace{.4in}\mbox{ and } \hspace{.4in}\tr^*(\mathcal{B}^+_x) \cong \lim_{\rightarrow}\left( \tr^*\left(\mathcal{M}_0^{\sigma^k(x)}\right),\sigma_* \right).
  \end{equation}

  Now consider the composition $\sigma_*\circ \sigma_* = \mathcal{T}^{\sigma(x)}_1\circ i^{\sigma(x)}_x\circ \mathcal{T}^x_1\circ i^x_1 :\mathcal{M}_0^x\rightarrow \mathcal{M}_0^{\sigma^2(x)}$. Since both $\tr(\mathcal{M}_1^{\sigma(x)})$ and $\tr(\mathcal{M}_2^x)$ are isomorphic to $\mathbb{C}^{|\mathcal{V}_2^+|}$ and there is a canonical correspondence between their bases $\{\tau_1,\dots, \tau_{|\mathcal{V}^+_2|}\}$ and $\{\tau_1',\dots, \tau'_{|\mathcal{V}^+_2|}\}$, respectively, we have that
$$\tau_\ell\left(i^{\sigma(x)}_1\left(\mathcal{T}^x_1\circ i^x_1(a)\right)\right) = \tau'_\ell\left(i^{x}_2  i^x_1(a)\right)$$
  for all $\ell\in \{1,\dots, |\mathcal{V}^+_2|\}$. So we can now write the composition in detail:
  \begin{equation*}
    \begin{split}
      \mathcal{T}^{\sigma(x)}_1\circ i^{\sigma(x)}_x\circ \mathcal{T}^x_1\circ i^x_1(a) &= \left(\tau_1\left(i^{\sigma(x)}_1\left(\mathcal{T}^x_1\circ i^x_1(a)\right)\right),\dots, \tau_{|\mathcal{V}^+_2|}\left(i^{\sigma(x)}_1\left(\mathcal{T}^x_1\circ i^x_1(a)\right)\right)\right) \\
      &= \left(\tau_1\left(i^{x}_2 i^x_1(a)\right),\dots, \tau_{|\mathcal{V}^+_2|}\left(i^{x}_2 i^x_1(a)\right)\right) \\
      &= \left(\left(i^{x}_2 i^x_1\right)^*\tau_1(a),\dots, \left(i^{x}_2 i^x_1\right)^*\tau_{|\mathcal{V}^+_2|}(a)\right) \\
      &= \left(\left(A_2A_1\right)^*\tau_1(a),\dots, \left(A_2 A_1\right)^*\tau_{|\mathcal{V}^+_2|}(a)\right),
    \end{split}
  \end{equation*}
where we have abused notation slightly in using $\tau_\ell$ to denote both the $\ell^{th}$ canonical trace in $\tr(\mathcal{M}_2^x)$ and the one in $\tr(\mathcal{M}_1^{\sigma(x)})$. This immediately generalizes to
  \begin{equation}
    \label{eqn:cocycle!}
\begin{split}
  \sigma_*^{(k)}:=\sigma_*\circ \cdots\circ  \sigma_*&:\mathcal{M}_0^x\rightarrow \mathcal{M}_0^{\sigma^k(x)} \hspace{.3in} \mbox{ defined by }\\
 a\mapsto & \left(\left(A_k\cdots A_1\right)^*\tau_1(a),\dots, \left(A_k\cdots A_1\right)^*\tau_{|\mathcal{V}^+_k|}(a)\right).
  \end{split}
  \end{equation}

  \subsection{The trace cocycle}
     \label{subsec:TrCocycle}
     Let $\mathcal{F}$ be a family of substitution rules on the set of prototiles $t_1,\dots, t_M$ and let $X_\mathcal{F}$ be the subshift that it defines.
\begin{definition}
  The \textbf{trace bundle} $p:\tr(\mathcal{F})\rightarrow X_\mathcal{F}$ is the bundle over $X_\mathcal{F}$ where $p^{-1}(x) = \tr(\mathcal{M}^x_0)$ for all $x\in X_\mathcal{F}$. The \textbf{cotrace bundle} $q:\tr^*(\mathcal{F})\rightarrow X_\mathcal{F}$ is the dual of the trace bundle, where $q^{-1}(x) = \tr^*(\mathcal{M}^x_0)$ for all $x\in X_\mathcal{F}$.
\end{definition}
\begin{definition}
The \textbf{trace cocycle} is the bundle map $\Theta:\tr^*(\mathcal{F})\rightarrow \tr^*(\mathcal{F})$ defined by $\Theta_x:(x,\tau')\mapsto (\sigma(x),\sigma_*(\tau'))$ for all $x\in X_\mathcal{F}$, $\tau'\in \tr^*(\mathcal{M}^x_0)$.
\end{definition}
Since $\tr^*(\mathcal{M}_0^x)$ is a finite dimensional vector space we endow it with a norm $\|\cdot \|$. Note that for all $y\in X_\mathcal{F}$ close enough to $x$ we will have $\tr^*(\mathcal{M}_0^x) = \tr^*(\mathcal{M}_0^y)$ and thus all these spaces inherit the same norm. With a norm in every space $\tr(\mathcal{M}_0^x)$, we now appeal to Oseledets theorem. Let $\|\cdot \|_{op}$ be the operator norm. Since the maps $\sigma_*$ can be singular but the base transformation $\sigma:X_\mathcal{F}\rightarrow X_\mathcal{F}$ invertible, we can appeal to the semi-invertible Oseledets theorem \cite{FLQ:semiOseledets} and obtain a decomposition of the trace spaces which is invariant under the dynamics.
\begin{theorem}[Semi-invertible Oseledets theorem \cite{FLQ:semiOseledets}]
  Let $\mathcal{F}$ be a family of substitution rules on $t_1,\dots, t_M$ tiles and $\mu$ a minimal and $\sigma$-invariant Borel ergodic probability measure on $X_\mathcal{F}$. Suppose that $\log^+ \|\sigma_*\|_{op}\in L^1_\mu$. Then there exist numbers $\lambda_1^\pm\geq \lambda_2^\pm\geq\cdots\geq \lambda_{r^\pm}^\pm$, where $\lambda^+_i> 0$ and $\lambda^-_i\leq 0$, such that for $\mu$-almost every $x$ there is a measurable, $\sigma_*$-invariant family of subspaces $V^\pm_j(x),V^\infty(x)\subset \tr^*(\mathcal{M}_0^{x})$ :
  \begin{enumerate}
  \item We have $\tr^*(\mathcal{B}^+_x) = E^+_x\oplus E^-_x$ where
    $$E^\pm_x = \bigoplus_{i=1}^{r^\pm} V^\pm_i(x), \,\,\,\,\,   \,\,\,\,\,\mathrm{ and }\,\,\,\,\, \tr^*(\mathcal{M}^{x}_0)  = \tr^*(\mathcal{B}^+_x) \oplus V^\infty(x),$$
  \item $\sigma_* V_j^\pm(x) = V_j^\pm(\sigma(x))$ and $\sigma_*V^\infty(x)\subset V^\infty(\sigma(x))$,
  \item For any $v^\pm\in V^\pm_i(x)$ and $v_0\in V^\infty(x)$ we have that
    $$ \lim_{n\rightarrow \infty}\frac{\log \| \sigma^{(n)}_* v^\pm\|}{n} = \lambda^\pm_i\,\,\,\,\mbox{ and }\,\,\,\,\lim_{n\rightarrow \infty}\frac{\log \|  \sigma^{(n)}_* v_0\|}{n} = -\infty.$$
  \end{enumerate}
\end{theorem}
The collection of numbers $\lambda^\pm_i$ associated to the measure $\mu$ are the \textbf{Lyapunov exponents} of $\mu$. The set of all exponents is the \textbf{Lyapunov spectrum of $\mu$}. Given an inaviant measure $\mu$ satisfying the hypotheses of Oseledets theorem, an \textbf{Oseledets-generic} or \textbf{Oseledets-typical} point is a point $x$ for which the conclusions of the theorem hold.

In (i) of the above theorem we have made the indentification of the cotrace space $\tr^*(\mathcal{B}^+_x)$ with subspace of $\tr^*(\mathcal{M}_0^x)$ which consists of vectors which are not in the kernel of $\sigma_*^{(k)}$ for all $k>0$. This is justified by (\ref{eqn:traceSpaces}). Thus the restriction of $\sigma_*$ to $\tr^*(\mathcal{B}^+_x)$ is the linear map on the cotrace space induced by the shift $\sigma$. There is an analogous, dual, invariant decomposition of $\tr(\mathcal{B}^+_x)$ as $\tr(\mathcal{B}^+_x) = T^+_x\oplus T^-_x$ where
    $$T^\pm_x = \bigoplus_{i=1}^{r^\pm} T^\pm_i(x), \,\,\,\,\,   \,\,\,\,\,\mathrm{ and }\,\,\,\,\, \tr(\mathcal{M}^{x}_0)  = \tr(\mathcal{B}^+_x) \oplus T^\infty(x),$$  

The rest of this section is devoted to defining, for Oseledets-typical points $x\in\Sigma_N$, a map $j_x^+:\tr^*(\mathcal{B}_x^+)\rightarrow LF(\mathcal{B}_x^+)$ and deducing its equivariant properties with respect to the renormalization dynamics, that is, with respect to the shift map $\sigma:\Sigma_N\rightarrow \Sigma_N$, which is given by (\ref{eqn:equivariance2}). These properties will be used in \S \ref{sec:erg} in the study of ergodic integrals.

Denote by $\{\tau_1,\dots, \tau_{|\mathcal{V}_0^+|}\}$ the standard basis of $\tr(\mathcal{M}_0^{x})$ and by $\{\delta_1,\dots, \delta_{|\mathcal{V}_0^+|}\}$ the dual basis for $\tr^*(\mathcal{M}_0^{x})$. Oseledets theorem above gives a canonical identification of $\tr^*(\mathcal{B}^+_x)$ with a subspace of $ \tr^*(\mathcal{M}_0^{x})$, so any cotrace in $\tr^*(\mathcal{B}^+_x)$ can be written as
\begin{equation}
\label{eqn:cotraceExp}
\tau^* = \sum_{i=1}^{|\mathcal{V}_0^+|} \beta_{i}(\tau^*)\delta_{i}\in \tr^*(\mathcal{B}^+_x)\subset \tr^*(\mathcal{M}_0^{x}).
\end{equation}
We now define a map
\begin{equation}
  \label{eqn:jMap}
  j_x^+:\tr^*(\mathcal{B}^+_x)\longrightarrow LF(\mathcal{B}^+_x)\subset AF(\mathcal{B}^+_x)
\end{equation}
as follows. For $\tau^*\in \tr^*(\mathcal{B}^+_x)$, the image $[a_{\tau^*}] = [j^+_x(\tau^*)]$ is defined through its representative in $\mathcal{M}_0^{x}$:
\begin{equation}
  \label{eqn:jMap2}
  j^+_x(\tau^*) = \left(\beta_1(\tau^*),\dots,\beta_{|\mathcal{V}_0^+|}(\tau^*)\right)\in \bigoplus_{i=1}^{|\mathcal{V}_0^+|}\mathbb{C}  = \mathcal{M}_0^{x},
\end{equation}
which is well-defined by the expression (\ref{eqn:cotraceExp}). We denote by $[a_{\tau^*}] = [j^+_x(\tau^*)]$ its class in $LF(\mathcal{B}^+_x)$. Note that by (\ref{eqn:cocycle!}), we have that
$$j^+_{\sigma(x)}(\sigma_* v)= \left(\tau_1(i^x_1j^+_x v),\dots, \tau_{|\mathcal{V}^+_1|}(i^x_1j^+_x v)\right)\in \mathbb{C}^{|\mathcal{V}_1^+|}=\mathcal{M}_0^{\sigma(x)},$$
where $\tau_\ell$ is the canonical generator for $\tr(M_{n_\ell})$, the trace space for the $\ell^{th}$ summand of the multimatrix algebra $\mathcal{M}_1^{x}$. In general, (\ref{eqn:cocycle!}) gives 
\begin{equation}
  \label{eqn:equivariance2}
  \begin{split}
    j^+_{\sigma^k(x)}(\sigma_*^{(k)} v) &=  \left(\tau_1(i^x_k\cdots i^x_1j^+_x (v)),\dots, \tau_{|\mathcal{V}^+_k|}(i^x_k\cdots i^x_1j^+_x (v))\right) \\
    &=  \left(\sigma_{(k)}^*\tau_1(j^+_x (v)),\dots, \sigma_{(k)}^* \tau_{|\mathcal{V}^+_k|}(j^+_x (v))\right) \in\mathbb{C}^{|\mathcal{V}^+_k|}= \mathcal{M}_0^{\sigma^k(x)},
    \end{split}
\end{equation}
where $\sigma_{(k)}^*$ is the dual to $\sigma^{(k)}_*$.
\section{Ergodic integrals}
\label{sec:erg}
This section is devoted to the proof of the main result of the paper, Theorem \ref{thm:main}. First, some necessary notions are introduced and some estimates derived. Then, in \S \ref{subsec:uppBnd1}, a proof of the upper bound (\ref{eqn:uppBnd1}) in Theorem \ref{thm:main} is derived. This is followed by the construction of special averaging sets in \S \ref{subsec:lowBnds} and a proof of (\ref{eqn:uppBnd2}) in Theorem \ref{thm:main}.

Throughout this section we assume that we are working with a minimal, ergodic $\sigma$-invariant Borel probability measure on $X_\mathcal{F}$, and that the collection $\mathcal{F}$ of substitutions are uniformly expanding and compatible. Throughout this section we also assume that $x\in X_\mathcal{F}$ is an Oseledets typical, Poincar\'e recurrent point. Let 
$$\mathring{X}_{\mathcal{B}_x}^0:=\left\{\bar{e} \in\mathring{X}_\mathcal{B}:\Delta_x(\bar{e})\in\mho_x\right\}.$$
\begin{definition}
Let $\mathcal{F}$ be a family of substitution tilings on the tiles $t_1,\dots, t_M$ and let $\Omega_x = \Delta_x(\mathring{X}_{\mathcal{B}_x})$ be the tiling space given by the minimal Bratteli diagram $\mathcal{B}_x$. A \textbf{spanning system of patches} for $\Omega_x$ is a collection $\Gamma = \{\Gamma_k\}_{k\geq 0}$ of sets of patches $\Gamma_k = \{\mathcal{P}_v\}_{v\in\mathcal{V}^+_k}$ with the following properties: for each $v\in\mathcal{V}_k^+$ there is a path $\bar{e}_v =(\bar{e}^-_v,\bar{e}^+_v)= (\dots, e_{-2},e_{-1},e_1,e_2,\dots)\in \mathring{X}_{\mathcal{B}_x}^0$ with $r(\bar{e}^+_v|_k )  = v$ and in that case $\mathcal{P}_v = \mathcal{P}_k(\bar{e}^+_v)$.
\end{definition}
A spanning system of patches gives a catalogue of all the supertiles in a given space. Along with this catalogue we can find a subset of the tiling space itself which corresponds to each of the patches in this catalogue. More specifically, given a spanning system of patches $\Gamma$ there is a corresponding system of \textbf{plaques}. For each patch $\mathcal{P}_v$ given by the system $\Gamma$, the corresponding plaque in $\Omega_x$ is
$$\mathcal{P}'_v:=\bigcup_{t\in\mathcal{P}_v}\varphi_t(\mathcal{T}_{\bar{e}_v})\subset \Omega_x.$$
We will denote by $X^\Gamma_\mathcal{B}\subset \mathring{X}_\mathcal{B}^0$ the set of paths parametrized by $\mathcal{V}^+$ which give the spanning system of patches $\Gamma$. 

Let $\mathcal{L}(\Omega_x)$ be the set of Lipschitz functions on $\Omega_x$ and for each $f\in \mathcal{L}(\Omega_x)$ denote by $L_f$ the Lipschitz constant. Given a spanning system of patches $\Gamma$ we define for $f\in\mathcal{L}(\Omega_x)$ and each $k\in\mathbb{N}$ the vector
\begin{equation}
  \label{eqn:vectors}
  \begin{split}
    V_\Gamma^k(f) &= \left(\int_{\mathcal{P}'_{v_1}}f\, d\bar{s},\dots,\int_{\mathcal{P}'{v_{|\mathcal{V}^+_k|}}}f\, d\bar{s}\right) \\
    &= \left(\int_{\mathcal{P}_{v_1}}f\circ \varphi_s\left(\mathcal{T}_{e^+_{v_1}}\right)\, ds,\dots,\int_{\mathcal{P}_{v_{|\mathcal{V}^+_k|}}}f\circ \varphi_s\left(\mathcal{T}_{e^+_{v_{|\mathcal{V}^+_k|}}}\right)\, ds\right)\in \mathbb{C}^{|\mathcal{V}^+_k|},
  \end{split}
\end{equation}
  where $d \bar{s}$ is the natural, leafwise volume form on $\Omega_x$. In words, the vectors are obtained by integrating the function $f$ along level-$k$ super tiles of all possible types, and we use the plaques $\mathcal{P}'_{v_i}$ given by the spanning system of patches. This will allow us to know how the function integrates along bigger and bigger orbits.

Since $\mathrm{dim}\,\tr^*(\mathcal{M}_0^{\sigma^k(x)}) =  |\mathcal{V}^+_k|$ there is a canonical isomorphism between $\tr^*(\mathcal{M}_0^{\sigma^k(x)})$ and $\mathbb{C}^{|\mathcal{V}^+_k|}$ taking the dual of the generator $\tau_{\ell_i}\in \tr(M_{\ell_i})$ to the $i^{th}$ standard basis vector in $\mathbb{C}^{|\mathcal{V}^+_k|}$ for all $i=1,\dots, |\mathcal{V}^+_k|$, where $\mathcal{M}_0^{\sigma^k(x)} = M_{\ell_1}\oplus\cdots\oplus M_{\ell_{|\mathcal{V}^+_k|}}$. As such, we can think of each $V^k_\Gamma(f,\bar{e})$ as an element of $\tr^*(\mathcal{M}_0^{\sigma^k(x)})$,
we can compare $V_\Gamma^{k+1}(f,\bar{e})$ with $\sigma_*V_\Gamma^k(f,\bar{e})$. The $i^{th}$ component of the difference is
\begin{equation}
  \label{eqn:approx1}
\begin{split}
&  \left| (V_\Gamma^{k+1}(f,\bar{e}) - \sigma_*V_\Gamma^k(f,\bar{e}))_i\right| = \\
 &\hspace{1in} \left|\int_{\mathcal{P}_{v_i}}f\circ \varphi_s\left(\mathcal{T}_{e^+_{v_i}}\right)\, ds - \sum_{    e'\in r^{-1}(v_i)}  \int_{\mathcal{P}_{s(e')}}f\circ \varphi_s\left(\mathcal{T}_{e^+_{s(e')}}\right)\, ds   \right|.
\end{split}
\end{equation}

Let $\varepsilon\in (0,\lambda^+_1)$. Since each patch $\mathcal{P}_v$ for $v\in\mathcal{V}_{k+1}^+$ is the union of patches given by level-$k$ supertiles, for any edge $e\in\mathcal{E}^+_{k+1}$ the transverse distance between the plaques $\mathcal{P}_{r(e)}'$ and $\mathcal{P}_{s(e)}'$ is
\begin{equation}
  \label{eqn:patchDis}
  d(\mathcal{P}_{r(e)}',\mathcal{P}_{s(e)}')\leq C_\varepsilon e^{-(\lambda_1^+-\varepsilon)k}
    \end{equation}
where the constant $C_\varepsilon$ is independent of $e$ and only depends on the family $\mathcal{F}, \mu$ and $\varepsilon$. For $v_i\in\mathcal{V}^+_k$ and $e\in r^{-1}(v_i)$, let
$$\mathcal{P}_{v_i,e}:= f^{-1}_{\bar{e}_v|_k}\circ f_e\left(\bigcup_{\substack{e'\in \mathcal{E}^+_{0,k-1}:\\r(e')=s(e)}}f_{e'}(t_{s(e')})\right).$$
As such, there are the decompositions of each $\mathcal{P}_{v_i}$ as patches tiled by level-$k-1$ supertiles:
\begin{equation}
  \label{eqn:patchDecomp}
  \mathcal{P}_{v_i} =  \bigcup_{e\in r^{-1}(v_i)}\mathcal{P}_{v_i,e}\hspace{1in}\mbox{ and }\hspace{1in}  \mathcal{P}_{v_i}' =  \bigcup_{e\in r^{-1}(v_i)}\mathcal{P}_{v_i,e}',
\end{equation}
so it follows that
\begin{equation}
  \label{eqn:approx2}
  \begin{split}
&  \int_{\mathcal{P}_{v_i}}f\circ \varphi_s\left(\mathcal{T}_{e^+_{v_i}}\right)\, ds - \sum_{    e'\in r^{-1}(v_i)}  \int_{\mathcal{P}_{s(e')}}f\circ \varphi_s\left(\mathcal{T}_{e^+_{s(e')}}\right)\, ds \\
    & \hspace{.8in}= \sum_{e\in r^{-1}(v_i)}\int_{\mathcal{P}_{v_i,e}}f\circ \varphi_s\left(\mathcal{T}_{e^+_{v_i}}\right)\, ds -  \int_{\mathcal{P}_{s(e)}} f\circ \varphi_s\left(\mathcal{T}_{e^+_{s(e)}}\right)\, ds .
  \end{split}
\end{equation}
Since both of the terms
$$\int_{\mathcal{P}_{v_i,e}}f\circ \varphi_s\left(\mathcal{T}_{e^+_{v_i}}\right)\, ds \hspace{.2in}\mbox{ and }\hspace{.2in} \int_{\mathcal{P}_{s(e)}} f\circ \varphi_s\left(\mathcal{T}_{e^+_{s(e)}}\right)\, ds $$
are integrating $f$ along pieces of leaves which correspond to the patches given by level-$(k-1)$ supertiles, the distance between these pieces is at most $C_\varepsilon e^{-(\lambda_1^+-\varepsilon)k}$, we can use the Lipschitz property to bound
\begin{equation}
  \label{eqn:approx3}
\left| \int_{\mathcal{P}_{v_i,e}}f\circ \varphi_s\left(\mathcal{T}_{e^+_{v_i}}\right)\, ds -  \int_{\mathcal{P}_{s(e)}} f\circ \varphi_s\left(\mathcal{T}_{e^+_{s(e)}}\right)\, ds \right| \leq L_fC_\varepsilon e^{-(\lambda_1^+-\varepsilon)k}\int_{\mathcal{P}_{v_i,e}}\left|f\circ \varphi_s\left(\mathcal{T}_{e^+_{v_i}}\right)\right|\, ds.
\end{equation}
for any $e\in r^{-1}(v_i)$. Returning to (\ref{eqn:approx1}) and using (\ref{eqn:patchDis})-(\ref{eqn:approx3}):
\begin{equation}
  \label{eqn:approx4}
  \begin{split}
    &    \left| (V_\Gamma^{k+1}(f,\bar{e}) - \sigma_*V_\Gamma^k(f,\bar{e}))_i\right| \leq \sum_{e\in r^{-1}(v_i)} \left| \int_{\mathcal{P}_{v_i,e}}f\circ \varphi_s\left(\mathcal{T}_{e^+_{v_i}}\right)\, ds -  \int_{\mathcal{P}_{s(e)}} f\circ \varphi_s\left(\mathcal{T}_{e^+_{s(e)}}\right)\, ds \right|  \\
    &\hspace{.3in}\leq \sum_{e\in r^{-1}(v_i)} L_f C_\varepsilon e^{-(\lambda^+_1-\varepsilon)k}\left|\int_{ \mathcal{P}_{v_i,e}}f\circ \varphi_s\left(\mathcal{T}_{e^+_{v_i}}\right)\right| \leq \sum_{e\in r^{-1}(v_i)} L_f C_\varepsilon e^{-(\lambda^+_1-\varepsilon)k}\mathrm{Vol}(\mathcal{P}_{s(e)})\|f\|_\infty \\
    &\hspace{1in}\leq \sum_{e\in r^{-1}(v_i)}  C_{f,\varepsilon} e^{-(\lambda_1^+-\varepsilon)k}e^{(\lambda^+_1+\varepsilon)k}\|f\|_\infty \leq C_{f,\varepsilon,\mathcal{F}}e^{2\varepsilon k},
  \end{split}
\end{equation}
where we have used that $\mathrm{Vol}(\mathcal{P}_{s(e)})\leq e^{(\lambda_1^++\varepsilon)k}$, which follows from the fact that $\mathrm{Vol}(\mathcal{P}_{s(e)})$ is roughly the number of tiles in the level-$k$ supertile $\mathcal{P}_{s(e)}$, which is exactly $\tau_{s(e)}(i_k\cdots i_1(\mathrm{Id}))$, and this is bounded by the largest growth rate of the trace cocycle.

By the estimate above we have that for any $\varepsilon>0$:
$$ \left\| V_\Gamma^{k+1}(f,\bar{e}) - \sigma_*V_\Gamma^k(f,\bar{e})\right\|\leq C^*_\varepsilon e^{2\varepsilon k},$$
for all $k>0$, so we can now invoke Bufetov's approximation Lemma \cite[Lemma 2.8]{bufetov:limit}, which says that, given a sequence of matrices $\{\Theta_k\}$ defined by a cocycle and sequence of vectors $\{V_k\}$ such that $\|V_{k+1} - \Theta_k V_k\|\leq C_\epsilon e^{\epsilon k}$ then there exists a vector $v_*$ on the first vector space whose orbit shadows the vectors $V_k$ at an exponential scale: $\|\Theta_k\cdots \Theta_1 v_* - V_{k+1}\|\leq C'_\epsilon e^{\epsilon k}$.

Applied to our situation, by (\ref{eqn:approx3}) and Bufetov's approximation Lemma there exists a $a_{f,\Gamma}\in E^+_x\subset \tr^*(\mathcal{B}_x^+)\subset \tr^*(\mathcal{M}_0^x) $ with the property that
\begin{equation}
  \label{eqn:BufLemma}
  \| j^+_{\sigma^{k+1}(x)}\left( \sigma_*^{(k)}a_{f,\Gamma}\right)  -  V^{k+1}_\Gamma(f)\|\leq C_\varepsilon'e^{2\varepsilon k}
\end{equation}
for all $k>0$. Thus we get a map
$$i^+_\Gamma:\mathcal{L}(\Omega_x)\rightarrow \tr^*(\mathcal{B}^+_x)$$
with $i^+_\Gamma(f) = a_{f,\Gamma}$ as defined above for any $f\in\mathcal{L}(\Omega_x)$. By composition with the map $j^+_x$ in (\ref{eqn:jMap}) we get a map $j^+_x\circ i^+_x:\mathcal{L}(\Omega_x)\rightarrow LF(\mathcal{B}^+_x)$. 
\subsection{Proof of the upper bound (\ref{eqn:uppBnd1})}
\label{subsec:uppBnd1}
For a tiling $\mathcal{T}$ of $\mathbb{R}^d$ of finite local complexity and a good Lipschitz domain $B$ with nonempty interior, we denote by $T\cdot B$ the set $(T\mathrm{Id})B$ and by $\mathcal{O}^-_\mathcal{T}(B)$ all the tiles of $\mathcal{T}$ which are completely contained in $B$.

Given $x\in\Sigma_N$, denote by $\theta_{(n)_x} = \theta_{x_{n}}\theta_{x_{n-1}}\cdots \theta_{x_{1}}$ the product of the contracting constants from the susbtitution maps. In other words, $\theta_{x_i}$ is the contraction constant of the substitution map $\mathcal{F}_{x_i}$. The following was proved in \cite[Lemma 8]{ST:random}.
\begin{lemma}
\label{lem:decomp}
  For a good Lipschitz domain $B$ with nonempty interior, tiling $\mathcal{T}_{\bar{e}}\in \Omega_x$ and $T>0$ there exists an integer $n = n(T,B)$ and a decomposition
  \begin{equation}
    \label{eqn:decomp2}
    \mathcal{O}_{\mathcal{T}_{\bar{e}}}^-(T\cdot B) =\bigcup_{i=0}^n\bigcup_{j=1}^{M(i)}\bigcup_{k=1}^{\kappa^{(i)}_j}t^{(i)}_{j,k}
  \end{equation}
  where $t^{(i)}_{j,k}$ is a level-$i$ supertile of type $j$ with
  \begin{enumerate}
  \item $\kappa^{(n)}_j\neq 0$ for some $j$ and $\mathrm{Vol}(T\cdot B)\leq K_1 \theta_{(n)_x}^{-d}$,
  \item $\displaystyle\sum_{j=1}^{M(i)}\kappa_j^{(i)}\leq K_2 \mathrm{Vol}(\partial (T\cdot B))\theta_{(i)_x}^{d-1}  $ for $i=0,\dots, n-1$.
  \item $R_1\theta_{(n)_x}^{-1} < T < R_1^{-1} R_2 \theta_{(n)_x}^{-1}$ and $n(R_2T, B)>n(T,B)$.
  \end{enumerate}
  for some $K_1,K_2, R_1, R_2>0$.
\end{lemma}
Let $B$ be Lipschitz domain and $T>1$. For $\mathcal{T}_{\bar{e}}\in\Omega_x$ consider a level-$i$ super tile $t^{(i)}_{j,k}$ of type $j$ given by the decomposition given in Lemma \ref{lem:decomp} and $f\in\mathcal{L}(\Omega_x)$. For any $\varepsilon>0$ and spanning system $\Gamma$, as in (\ref{eqn:approx4}), one has that
\begin{equation}
  \label{eqn:BufApprox}
  \left| \int_{t^{(i)}_{j,k}} f\circ \varphi_s(\mathcal{T}_{\bar{e}})\, ds - \int_{\mathcal{P}_{v_j}}f\circ \varphi_s(\mathcal{T}_{\bar{e}_{v_j}})\, ds\right|\leq C_\varepsilon L_fe^{2\varepsilon i}
  \end{equation}
with $v_j\in\mathcal{V}^+_i$. Combining this with (\ref{eqn:BufLemma}) we have that
\begin{equation}
  \label{eqn:superEstimate}
  \left|\int_{t^{(i)}_{j,k}} f \circ \varphi_s(\mathcal{T})\, ds  - j^+_{\sigma^i(x)}\left(\sigma_*^{(i)}i^+_\Gamma(f)\right)_j\right|\leq C_\varepsilon'' e^{2\varepsilon i},
\end{equation}
where $C''_\varepsilon$ only depends on $\varepsilon$ and $\mathcal{F}$.

For any Oseledets regular $x$ and a generating trace $\tau_\ell^{x,k} \in \tr(\mathcal{M}_0^{\sigma^k(x)})$ there is a decomposition
$$\tau_\ell^{x,k} = \sum_{m^\pm=1}^{d^\pm_\mu} b_{m^\pm,\ell,k}\tau^\pm_{m,k} + b_{\infty, \ell,k}\tau_{\infty,k},$$
where $\tau^\pm_{m,k}\in T^\pm_m(\sigma^k(x))$ and $\tau_{\infty,k} \in T^\infty(\sigma^k(x))$ are unit vectors. Note that in such decomposition there is a $\mathcal{N}$ such that $|b_{m^\pm,\ell,k}|\leq \mathcal{N}$ for all indices. This follows from the fact that $\tau_{m,k}^\pm$ are unit vectors, $\tau_\ell^{x,k}$ are generating traces (i.e. unit vectors), and we are dealing with finite dimensional vector spaces. Since $i^+_\Gamma(f)\in E^+_x$, using (\ref{eqn:decomp2}) and (\ref{eqn:equivariance2}) it follows that
\begin{equation}
  \label{eqn:ergInt1} 
\begin{split}
 & \hspace{.3in}\int_{\mathcal{O}_{\mathcal{T}_{\bar{e}}}^-(T\cdot B)}f\circ \varphi_s(\mathcal{T}_{\bar{e}})\, ds  = \sum_{i=0}^{n(T)}\sum_{j=1}^{M(i)}\sum_{k=1}^{\kappa^{(i)}_j} \int_{t^{(i)}_{j,k}} f\circ \varphi_s(\mathcal{T}_{\bar{e}})\, ds \\
  & \hspace{1.5in}= \sum_{i=0}^{n(T)}\sum_{j=1}^{M(i)} \kappa^{(i)}_j\left(j^+_{\sigma^i(x)}\left(\sigma_*^{(i)}i^+_\Gamma(f)\right)_j + O(e^{2\varepsilon i}) \right)  \\
  & \hspace{1 in} =\sum_{i=0}^{n(T)}\sum_{j=1}^{M(i)} \kappa^{(i)}_j\left(\sigma^*_{(i)}\tau_j^{x,i} \left( j^+_{x}\left(i^+_\Gamma(f)\right) \right)+ O(e^{2\varepsilon i}) \right) \\
  & \hspace{.5 in}= \sum_{i=0}^{n(T)}\sum_{j=1}^{M(i)} \kappa^{(i)}_j\left(
  \sum_{m=1}^{d^+_\mu} b_{m,\ell,i} \sigma^*_{(i)} \tau^+_{m,i}\left(j^+_{x}(i^+_\Gamma(f))\right) 
  + O(e^{2\varepsilon i}) \right) \\
    & =   \sum_{i=0}^{n(T)}\sum_{j=1}^{M(i)} \kappa^{(i)}_j\left(
  \sum_{m=1}^{d^+_\mu} b_{m,\ell,i} \left\|\sigma_*^{(i)}|_{V^+_m(x)}\right\| \tau^+_{m,0}\left(j^+_{x}(i^+_\Gamma(f))\right) 
   + O(e^{2\varepsilon i}) \right).
\end{split}
\end{equation}
For any $\varepsilon>0$, the bounds in Lemma \ref{lem:decomp} give
\begin{equation}
\label{eqn:ergInt2}
  \begin{split}
    &\left|\int_{\mathcal{O}_{\mathcal{T}_{\bar{e}}}^-(T\cdot B)}f\circ \varphi_s(\mathcal{T}_{\bar{e}})\, ds\right|\leq \sum_{i=0}^{n(T)}\sum_{j=1}^{M(i)} \kappa^{(i)}_j\left( C_1\sum_{m=1}^{d^+_\mu}  \tau^+_{m,0}\left(j^+_{x}(i^+_\Gamma(f))\right) e^{(\lambda_m^++\varepsilon)i} \right)\\
    &    \hspace{1in}\leq C_2 \sum_{i=0}^{n(T)}\mathrm{Vol}(\partial (T\cdot B))\theta_{(x)_i}^{d-1}\sum_{m=1}^{d^+_\mu}  \tau^+_{m,0}\left(j^+_{x}(i^+_\Gamma(f))\right) e^{(\lambda_m^++\varepsilon)i} \\
 &    \hspace{1in}\leq C_3  \sum_{i=0}^{n(T)}(\theta_{x_{i+1}}\cdots \theta_{x_n})^{1-d}\sum_{m=1}^{d^+_\mu}  \tau^+_{m,0}\left(j^+_{x}(i^+_\Gamma(f))\right) e^{(\lambda_m^++\varepsilon)i}
  \end{split}
\end{equation}
where the fact that $\mathrm{Vol}(\partial (T\cdot B))(\theta_{e_1}\cdots \theta_{e_i})^{d-1}\leq C_3 (\theta_{x_{i+1}}\cdots \theta_{x_n})^{1-d} $ was used. This last estimate is a straightforward consequence of the estimates in Lemma \ref{lem:decomp} and the fact that $\mathrm{Vol}(\partial(T\cdot B ))\sim \mathrm{Vol}(T\cdot B)^{\frac{d-1}{d}}$ for Lipschitz domains $B$ and large $T$. If $\tau_{m,0}^+(j^+_x(i^+_\Gamma(f))) = 0$ for all $m=1,\dots, r-1$ but $\tau_{r,0}^+(j^+_x(i^+_\Gamma(f))) \neq 0$ for some $r\leq d^+_\mu$, then
\begin{equation}
\label{eqn:ergInt3}
\left|\int_{\mathcal{O}_{\mathcal{T}_{\bar{e}}}^-(T\cdot B)}f\circ \varphi_s(\mathcal{T}_{\bar{e}})\, ds\right|\leq  C_4  \sum_{i=0}^{n(T)}(\theta_{x_{i+1}}\cdots \theta_{x_n})^{1-d}  \tau^+_{r,0}\left([a_f]\right) e^{(\lambda_r^++\varepsilon)i} .
\end{equation}
Now, for any $\varepsilon>0$ we have that
$$(\theta_{x_{i+1}}\cdots\theta_{x_n})^{1-d} \leq C_\varepsilon' e^{\left(\frac{\lambda_1^+}{d}+\varepsilon )(n-i)(d-1\right)}.$$
for some $C_\varepsilon'>0$. Indeed, for an Oseledets-typical $x\in X_\mathcal{F}$, the leading exponent $\lambda_1^+$ gives the exponential rate of increase of the number of paths starting from $\mathcal{V}^+_0$ of length $k>0$ in $\mathcal{B}^+_x$. Since the paths of length $k$ are in bijection with tiles in $k$-approximants, the number of paths of length $k$ also give estimates on the volumes of patches for level-$k$ supertiles. Thus $\lambda_1^+$ gives the exponential rate of increase of volume of supertiles. So
\begin{equation}
  \label{eqn:asympContr}
  \lim_{n\rightarrow \infty} \frac{\log \theta_{x_1}\cdots \theta_{x_n}}{n} = -\frac{\lambda_1^+}{d}.
\end{equation}
Therefore we can continue with (\ref{eqn:ergInt2}):
\begin{equation}
  \label{eqn:ergLong}
  \begin{split}
&    \left|\int_{\mathcal{O}_{\mathcal{T}_{\bar{e}}}^-(T\cdot B)}f\circ \varphi_s(\mathcal{T}_{\bar{e}})\, ds\right|\leq C_5  \sum_{i=0}^{n(T)} e^{\left(\frac{\lambda_1^+}{d} + \varepsilon\right)(n-i)(d-1)} e^{(\lambda_r^++\varepsilon)i} =C_5  \sum_{i=0}^{n(T)} e^{\left(\frac{\lambda_1^+}{d} + \varepsilon\right)(i-n)(1-d)} e^{(\lambda_r^++\varepsilon)i} \\
    & \hspace{.8in}= C_5\sum_{i=0}^{n(T)} \exp\left[ \left( \lambda_r^+ - \lambda_1^+\frac{d-1}{d} + \varepsilon(2-d)\right)i +\left( \lambda_1^+\frac{d-1}{d}+\varepsilon(d-1)\right)n \right] \\
    & \hspace{.6in}= C_5\exp\left[ \left( \lambda_1^+\frac{d-1}{d}+\varepsilon(d-1)\right)n\right]\sum_{i=0}^{n(T)} \exp\left[ \left( \lambda_r^+ - \lambda_1^+\frac{d-1}{d} + \varepsilon(2-d)\right)i  \right] \\
    & \hspace{.4in}= C_6\exp\left[ \left( \lambda_1^+\frac{d-1}{d}+\varepsilon(d-1)\right)n\right] \left| 1-  \exp\left[ \left( \lambda_r^+ - \lambda_1^+\frac{d-1}{d} + \varepsilon(2-d)\right)n+1  \right] \right| \\
    & \hspace{.2in}\leq C_7\exp  \left[ \max \left\{  \lambda_1^+\frac{d-1}{d} +  \varepsilon(d-1), \lambda_r^+ +  \varepsilon   \right\}n \right].
  \end{split}
\end{equation}
Defining
$$\bar{\lambda}_{r,\varepsilon}^+ := \max \left\{  \lambda_1^+\frac{d-1}{d} +  \varepsilon(d-1),  \lambda_r^+ +  \varepsilon   \right\}$$
and using (iii) from Lemma \ref{lem:decomp}, we have that
\begin{equation}
  \label{eqn:ergInt6}
  \begin{split}
    \frac{\log \left|\displaystyle\int_{\mathcal{O}_{\mathcal{T}_{\bar{e}}}^-(T\cdot B)}f\circ \varphi_s(\mathcal{T}_{\bar{e}})\, ds\right|}{\log T} &\leq \frac{\log\left( C_7  \right) +\bar{\lambda}_{r,\varepsilon}^+n(T)}{\log(R_1^{-1}) + \log \theta_{(x)_n}^{-1}}.  \\
    \end{split}
\end{equation}
Recall that by (\ref{eqn:asympContr}) we have that
$$\lim_{n\rightarrow \infty}\frac{\log \theta_{(x)_n}^{-1}}{n} = \lim_{n\rightarrow \infty}\frac{1}{n}\sum_{i=1}^n\log\theta_{x_i}^{-1} = \frac{\lambda_1^+}{d},$$
so it follows from (\ref{eqn:ergInt6}) that
\begin{equation}
  \label{eqn:ergInt4}
\begin{split}
  \limsup_{T\rightarrow \infty}\frac{\log \left|\displaystyle\int_{\mathcal{O}_{\mathcal{T}_{\bar{e}}}^-(T\cdot B)}f\circ \varphi_s(\mathcal{T}_{\bar{e}})\, ds\right|}{\log T}  &\leq \limsup_{n\rightarrow \infty} \frac{\log\left( C_7  \right) +\bar{\lambda}_{r,\varepsilon}^+n(T)}{\log(R_1^{-1}) + \log \theta_{(x)_n}^{-1}}\\
  &= \frac{\bar{\lambda}_{r,\varepsilon}^+}{\lambda_1^+}d.
\end{split}
\end{equation}
Now, since
\begin{equation}
  \label{eqn:ergInt5}
  \begin{split}
    \left|\int_{T\cdot B}f\circ \varphi_s(\mathcal{T}_{\bar{e}})\, ds\right| &\leq \left|\int_{\mathcal{O}_{\mathcal{T}_{\bar{e}}}^-(T\cdot B)}f\circ \varphi_s(\mathcal{T}_{\bar{e}})\, ds\right| + \left|\int_{T\cdot B - \mathcal{O}_{\mathcal{T}_{\bar{e}}}^-(T\cdot B)}f\circ \varphi_s(\mathcal{T}_{\bar{e}})\, ds\right| \\
    &\leq  C_8  e^{\bar{\lambda}_{r,\varepsilon}^+n(T)} + O(\partial (T\cdot B)) =  C_8  e^{\bar{\lambda}_{r,\varepsilon}^+n(T)} + O(T^{d-1}),
    \end{split}
\end{equation}
this completes the proof of the bound (\ref{eqn:uppBnd1}).
\subsection{Special averaging sets and proof of (\ref{eqn:uppBnd2})}
\label{subsec:lowBnds}
Let $x\in X_\mathcal{F}$ be a Poincar\'e-recurrent, Oseledets-regular point and $\mathcal{T}_{\bar{e}}\in \Omega_x$ for some $\bar{e} = (\bar{e}^-,\bar{e}^+) = (\dots, e_{-2},e_{-1},e_1,e_2,\dots)\in\mathring{X}_{\mathcal{B}_x}$. For any $\varepsilon>0$ there exists a $T_\varepsilon>0$ such that
\begin{itemize}
\item  $T^{-1}\cdot \mathcal{O}_{\mathcal{T}_{\bar{e}}}^-(T\cdot B)$ is $\varepsilon$-close to $B$ in the Hausdorff metric,
\item $\mathcal{O}_{\mathcal{T}_{\bar{e}}}^-$ contains a ball of radius twice the minimal radius so that every ball of such radius contains a copy of every prototile in its interior
 \end{itemize}
for all $T>T_\varepsilon$. Pick some $T_* > T_\varepsilon$ and define $B_\varepsilon = T_*^{-1}\cdot \mathcal{O}_{\mathcal{T}_{\bar{e}}}^-(T_*\cdot B)\mbox{ and }\mathcal{P}_\varepsilon(\bar{e}) = T_*B_\varepsilon=\mathcal{O}_{\mathcal{T}_{\bar{e}}}^-(T_*\cdot B)$, which is a patch for all tilings in $\Omega_x$. The set $B_\varepsilon$ is at most $\varepsilon$ close to $B$ in the Hausdorff metric.

Let $k_i\rightarrow\infty$ denote the recurrence times, $\sigma^{k_i}(x)\rightarrow x$, and suppose that $\sigma^{k_i}(\bar{e})$ converges to $\bar{e}^* =  (\dots, e_{-2}^*,e_{-1}^*,e_1^*,e_2^*,\dots)\in \mathring{X}_{\mathcal{B}_x}$ along these times. Let $k_x$ be the smallest integer so that for all $v\in\mathcal{V}^+_0$ and $w\in\mathcal{V}^+_{k_x}$ there is a path $p\in\mathcal{E}^+_{0,k_x}$ with $s(p)=v$ and $r(p)=w$. It follows that there is a $k_\varepsilon'\geq k_x$ and finite set of paths $\mathcal{E}'_{B_\varepsilon}\subset \mathcal{E}^+_{0,k_\varepsilon'}$ such that for all $p\in\mathcal{E}'_{B_\varepsilon}$ one has that $r(p)=r(\bar{e}|_{k_\varepsilon'})$ and such that the patch $\mathcal{P}_\varepsilon(\bar{e})$ decomposes as
\begin{equation}
  \label{eqn:ePatchDec}
  \mathcal{P}_\varepsilon(\bar{e}) = \varphi_{\tau_{\bar{e}}}\left( \bigcup_{\bar{e}'\in \mathcal{E}'_{B_\varepsilon}} f^{-1}_{\bar{e}|_{k_\varepsilon'}}\circ f_{\bar{e}'}(t_{s(\bar{e}')}) \right)\subset \varphi_{\tau_{\bar{e}}}\left( \mathcal{P}_{k'_\varepsilon}(\bar{e}^+)\right),
\end{equation}
where $\tau_{\bar{e}}\in\mathbb{R}^d$ is completely determined by the negative part of $\bar{e}$. By the choice of $T_*$ the patch $\mathcal{P}_\varepsilon(\bar{e})$ is decomposed as the union of tiles
\begin{equation}
\label{eqn:tileDecomp}
  \mathcal{P}_\varepsilon(\bar{e}) = \bigcup_{\ell=1}^{M} \bigcup_{j=1}^{\kappa(\ell)}t_{\ell,j}
\end{equation}
where $t_{\ell,j}$ is a translate of the prototilie $t_\ell$. Note that the number of tiles in the decomposition (\ref{eqn:tileDecomp}) is $|\mathcal{E}'_{B_\varepsilon}|$ from (\ref{eqn:ePatchDec}).

 By minimality, there is a smallest $k_\varepsilon> k_\varepsilon'$ such that there is a path $p'\in\mathcal{E}_{k_\varepsilon',k_\varepsilon}$ with $s(p') = r(\bar{e}|_{k_\varepsilon'})$ and $r(p') = r(\bar{e}^*|_{k_\varepsilon})$. This gives a finite set of paths $\mathcal{E}_{B_\varepsilon}\subset \mathcal{E}^+_{0,k_\varepsilon}$ obtained by concatenating $p'$ to every path $e'\in\mathcal{E}'_\varepsilon$. As such, the patch decomposes as
$$\mathcal{P}_\varepsilon(\bar{e}) = \varphi_{\tau_{\bar{e}}}\left( \bigcup_{\bar{e}'\in \mathcal{E}_{B_\varepsilon}} f^{-1}_{\bar{e}|_{k_\varepsilon'}p'}\circ f_{\bar{e}'}(t_{s(\bar{e}')}) \right).$$
Considering the patch
$$\mathcal{P}^*_\varepsilon(\bar{e}) =  \bigcup_{\bar{e}'\in \mathcal{E}_{B_\varepsilon}} f^{-1}_{\bar{e}^*|_{k_\varepsilon}}\circ f_{\bar{e}'}(t_{s(\bar{e}')}) \subset \mathcal{P}_{k_\varepsilon}(\bar{e}^{*+}) ,$$
by Lemma \ref{lemma:deform}, there is a $\tau_{\bar{e},\bar{e}^*}\in\mathbb{R}^d$ such that $\mathcal{P}^*_\varepsilon(\bar{e}) = \mathcal{P}_\varepsilon(\bar{e}) + \tau_{\bar{e},\bar{e}^*}$.

Let me take the time here to describe what is about to be done. So far we have constructed a set $B_\varepsilon$ which is $\varepsilon$-close to $B$, but it is of a special type: when dilated by $T_*$, it becomes a patch which has been denoted by $\mathcal{P}_\varepsilon(\bar{e})$. Now, since $x$ is Poincar\'e recurrent, there is a sequence of times $k_i\rightarrow \infty$ such that all the tilings in $\Omega_{\sigma^{k_i}(x)}$ admit $\mathcal{P}_\varepsilon(\bar{e})$ as a patch since $\sigma^{k_i}(x)\rightarrow x$. Recall that by Proposition \ref{prop:renorm} patches in $\Omega_{\sigma^{k_i}(x)}$ correspond to ``superpatches'' in $\Omega_x$, that is, patches in $\Omega_x$ made up of level-$k_i$ supertiles. So we want to dilate $\mathcal{P}_\varepsilon(\bar{e})$ along a sequence of times $T_i$ so that, up to a small translation, it becomes a patch made up \textbf{ of only} level-$k_i$ supertiles, unlike general dilations of sets which, as Lemma \ref{lem:decomp} shows, invlolve supertiles of all levels. We do all this because the integrals along this sequence of superpatches can be controlled very well.

For all $i$ large enough the set $\mathcal{E}^+_{k_i,k_i+k_\varepsilon}$ is a copy of $\mathcal{E}^+_{0,k_\varepsilon}$ and as such it contains a copy $\mathcal{E}_{B_{\varepsilon}}^i$ of $\mathcal{E}_{B_\varepsilon}$. In other words, since $\mathcal{E}^+_{k}$ is determined by $\mathcal{F}_{x_k}$ and $x_{k_i+j} = x_j$ for all large $i$ and $0< j \leq k_\varepsilon$ (by Poincar\'e recurrence), we can make the identification $\mathcal{E}^+_{k_i+j} = \mathcal{E}^+_{j}$ for all large $i$. Moreover, since $k_\varepsilon\geq k_x$, for $i$ large enough there is a path from $v$ to $r(\bar{e}_{k_i+k_\varepsilon})$ for all $v\in \mathcal{V}^+_{k_i}$. Define the patches
\begin{equation}
  \label{eqn:PiPatch}
  \mathcal{P}^i_\varepsilon(\bar{e}):= \bigcup_{\substack{e_v\in\mathcal{E}^{i}_{B_\varepsilon} \\ \bar{e}'\in\mathcal{E}_{0,k_i} \\ s(e_v) = r(\bar{e}')}} f^{-1}_{\bar{e}|_{k_i+k_\varepsilon}}\circ f_{\bar{e}' e_v}(t_{s(\bar{e}')}) 
  \end{equation}
and note that
\begin{equation}
  \label{eqn:bigBlowup}
  \begin{split}
    \mathcal{P}^i_\varepsilon(\bar{e}) &=  \bigcup_{\substack{e_v\in\mathcal{E}^{i}_{B_\varepsilon} \\ \bar{e}'\in\mathcal{E}_{0,k_i} \\ s(e_v) = r(\bar{e}')}} f^{-1}_{\bar{e}|_{k_i+k_\varepsilon}}\circ f_{\bar{e}' e_v}(t_{s(\bar{e}')}) = \bigcup_{\substack{e_v\in\mathcal{E}^{i}_{B_\varepsilon} \\ \bar{e}'\in\mathcal{E}_{0,k_i} \\ s(e_v) = r(\bar{e}')}} f^{-1}_{\bar{e}|_{k_i+k_\varepsilon}}\circ f_{e_v}\circ f_{\bar{e}' }(t_{s(\bar{e}')}) \\
    &=   \bigcup_{\substack{e_v\in\mathcal{E}^{i}_{B_\varepsilon} \\ \bar{e}'\in\mathcal{E}_{0,k_i} \\ s(e_v) = r(\bar{e}')}} f^{-1}_{\bar{e}|_{k_i}}\circ
    f^{-1}_{(e_{k_i+1},\dots, e_{k_i+k_\varepsilon})} \circ f_{e_v}\circ f_{\bar{e}' }(t_{s(\bar{e}')}) \\
    &=   f^{-1}_{\bar{e}|_{k_i}} \left(\bigcup_{\substack{e_v\in\mathcal{E}^{i}_{B_\varepsilon} \\ \bar{e}'\in\mathcal{E}_{0,k_i} \\ s(e_v) = r(\bar{e}')}}  f^{-1}_{\bar{e}^*|_{k_\varepsilon}} \circ f_{e_v}\circ f_{\bar{e}' }(t_{s(\bar{e}')}) \right) =  f^{-1}_{\bar{e}|_{k_i}} \left(\bigcup_{\substack{e_v\in\mathcal{E}_{B_\varepsilon} }}  f^{-1}_{\bar{e}^*|_{k_\varepsilon}} \circ f_{e_v}(t_{s(e_v)}) \right) \\
    & = f^{-1}_{\bar{e}|_{k_i}}( \mathcal{P}^*_\varepsilon(\bar{e})) = f^{-1}_{\bar{e}|_{k_i}}\left( \mathcal{P}_\varepsilon(\bar{e}) + \tau_{\bar{e},\bar{e}^*}\right) = f^{-1}_{\bar{e}|_{k_i}}\left( T_*\cdot B_\varepsilon + \tau_{\bar{e},\bar{e}^*}\right).
  \end{split}
  \end{equation}
Thus, setting $t_{\ell,j}^{(i)}:= f^{-1}_{\bar{e}|_{k_i}}(t_{\ell,j} + \tau_{\bar{e},\bar{e}^*})$ by (\ref{eqn:tileDecomp}) it follows that
$$\mathcal{P}_{\varepsilon}^i(\bar{e}) =  \bigcup_{\ell=1}^{M} \bigcup_{j=1}^{\kappa(\ell)} t^{(i)}_{\ell,j}$$
which expresses $\mathcal{P}^i_{\varepsilon}(\bar{e})$ as the union of level-$k_i$ supertiles of $\mathcal{T}_{\bar{e}}$.

\begin{lemma}
  \label{lem:vectors}
There is a compact set $\mathcal{K}\subset \mathbb{R}^d$ such that for all $i$ large enough there exists a $T_{i}>0$ and $\tau_i\in\mathcal{K}$ such that $T_i\cdot (\mathcal{P}_\varepsilon(\bar{e}) + \tau_i) = \mathcal{P}^i_x(\bar{e})$.
\end{lemma}
\begin{proof}
  By (\ref{eqn:bigBlowup}) we have that
  $$\mathcal{P}^i_\varepsilon(\bar{e}) = f^{-1}_{\bar{e}|_{k_i}}\left( T_*\cdot B_\varepsilon + \tau_{\bar{e},\bar{e}^*}\right) = \theta_{(k_i)_x}^{-1}T_*\cdot B_\varepsilon + \tau^i_{\bar{e},\bar{e}^*}$$
  for some $\tau^i_{\bar{e},\bar{e}^*}\in \mathbb{R}^d$.  Defining $T_i:= \theta_{(k_i)_x}^{-1}T_*$ we have that $\mathcal{P}^i_\varepsilon(\bar{e}) = T_i\cdot B_\varepsilon + \tau^i_{\bar{e},\bar{e}^*}$. By our assumption of recurrence we have that there exists a $R_\varepsilon$ such that for all $i$ large enough, the patch $\mathcal{P}_\varepsilon(\bar{e})$ is found in any ball of radius $R_\varepsilon$ around any point in $\mathbb{R}^d$ for any $\mathcal{T}\in \Omega_{\sigma^{k_i}(x)}$ for all $i$ large enough. The scaling $T_i$ relates the scales of $\Omega_x$ and those of $\Omega_{\sigma^{k_i}(x)}$. In other words, level-$k_i$ supertiles in $\Omega_{x}$ correspond to tiles in $\Omega_{\sigma^{k_i}(x)}$ and the difference in scales is precisely $T_i$. By this relationship of scale and repetitivity, any ball of radius $T_iR_\varepsilon$ in $\mathcal{T}_{\sigma^{k_i}(\bar{e})}$ contains a copy of the patch $\mathcal{P}^i_\varepsilon(\bar{e})$. So without loss of generality we can assume that $\tau^i_{\bar{e},\bar{e}^*}\in B_{T_iR_\varepsilon}(0)$. Letting $\tau_i = T_i^{-1}\tau^i_{\bar{e},\bar{e}^*}$ we get that $T_i\cdot (\mathcal{P}_\varepsilon(\bar{e}) + \tau_i) = \mathcal{P}^i_x(\bar{e})$.
\end{proof}
\subsubsection{Implicit upper bound}
Let $f\in\mathcal{L}(\Omega_x)$. By (\ref{eqn:bigBlowup}) for all $i$ large enough there is the decomposition
$$\int_{\mathcal{P}_\varepsilon^i(\bar{e})} f\circ \varphi_s(\mathcal{T}_{\bar{e}})\, ds = \sum_{\ell = 1}^{M}\sum_{j=1}^{\kappa(\ell)}\int_{t_{\ell,j}^{(i)}} f\circ \varphi_s(\mathcal{T}_{\bar{e}})\, ds$$
which, after choosing $\varepsilon'>0$ and using (\ref{eqn:BufApprox}) and (\ref{eqn:superEstimate}), becomes, as in (\ref{eqn:ergInt1})
\begin{equation}
\label{eqn:uppSplit}
  \begin{split}
    \int_{\mathcal{P}_\varepsilon^i(\bar{e})} f\circ \varphi_s(\mathcal{T}_{\bar{e}})\, ds &= \sum_{\ell = 1}^{M}\kappa(\ell) \left(j^+_{\sigma^{k_i}(x)}\left(\sigma_*^{(k_i)}i^+_\Gamma(f)\right)\right)_\ell + O(e^{2\varepsilon' k_i}) \\
   &= \sum_{\ell = 1}^{M}\kappa(\ell)  \sigma^*_{(k_i)}\tau_j^{x,k_i}  \left(j^+_{x}(i^+_\Gamma(f))\right)     + O(e^{2\varepsilon' k_i}) \\
   &= \sum_{\ell = 1}^{M}\kappa(\ell)   \sum_{m=1}^{d^+_\mu} b_{m,\ell,k_i} \sigma^*_{(k_i)} \tau^+_{m,k_i}\left(j^+_{x}i^+_\Gamma(f)\right)      + O(e^{2\varepsilon' k_i}) \\
   &= \sum_{\ell = 1}^{M}\kappa(\ell)   \sum_{m=1}^{d^+_\mu} b_{m,\ell,k_i}   \left\|\sigma_*^{(k_i)}|_{V^+_m(x)}\right\|  \tau^+_{m,0}\left(j^+_{x}(i^+_\Gamma(f))\right)        + O(e^{2\varepsilon' k_i}).
  \end{split}
\end{equation}
So if $\tau_{m,0}^+\left(j^+_{x}(i^+_\Gamma(f))\right) = 0$ for all $m<r$ but $\tau_{r,0}^+\left(j^+_{x}(i^+_\Gamma(f))\right) \neq 0$,
\begin{equation}
\label{eqn:lowBnd}
  \int_{\mathcal{P}_\varepsilon^i(\bar{e})} f\circ \varphi_s(\mathcal{T}_{\bar{e}})\, ds =   \sum_{\ell = 1}^{M}\kappa(\ell)   \sum_{m=r}^{d^+_\mu} b_{m,\ell,k_i}  \left\|\sigma_*^{(k_i)}|_{V^+_m(x)}\right\|  \tau^+_{m,0}\left([a_f]\right)        + O(e^{2\varepsilon' k_i})
\end{equation}
for all $i$, from which it follows that
$$\left|   \int_{T_i(B_\varepsilon + \tau_i)} f\circ \varphi_s(\mathcal{T}_{\bar{e}})\, ds \right|\leq C e^{(\lambda^+_r +\varepsilon') k_i}$$
for all $i$. Since $T_i$ is proportional to $\theta^{-1}_{(k_i)_x}$, we can estimate as in (\ref{eqn:ergInt3})-(\ref{eqn:ergInt4}) to obtain
$$\limsup_{i\rightarrow \infty}\frac{\log\left|\displaystyle\int_{T_i(B_\varepsilon + \tau_i)}f\circ \varphi_s(\mathcal{T}_{\bar{e}})\, ds\right|}{\log T_i} \leq \frac{\lambda_r^++\varepsilon'}{\lambda_1^+}d.$$
\subsubsection{Implicit lower bound}
We partition the set of indices $\{1,\dots, M\}$ into two sets $I^+,I^0$. An index $\ell$ is in $I^+$ if
$$\limsup_{i\rightarrow \infty} \frac{\left|\left(j^+_{\sigma^{k_i}(x)}\left(\sigma_*^{(k_i)}i^+_\Gamma(f)\right)\right)_\ell\right|}{\left\|\sigma^{k_i}_*\pi^+_{\ell,x}i^+_\Gamma f\right\|}>0,$$
and $\ell\in I^0$ otherwise, where we recall $\pi^+_{\ell,x}:\tr^*(\mathcal{M}^x_0)\rightarrow V_\ell^+(x) $ is the corresponding projection to the $\ell^{th}$ positive Oseledets subspace. The set $I^+$ is not empty because 1) by assumption, $\tau_\ell^+([a_f]) = 0$ for all $\ell<r$ and $\tau_r^+([a_f])\neq 0$, meaning that $\pi^+_{\ell,x}i^+_\Gamma f = 0$ for all $\ell<r$ but $\pi^+_{r,x}i^+_\Gamma f\neq 0$; and 2) all norms are equivalent in finite dimensional vector spaces.

Now we recall (\ref{eqn:uppSplit}) and express it with indices according to the partition $I^+,I^0$:
$$\int_{\mathcal{P}_\varepsilon^i(\bar{e})} f\circ \varphi_s(\mathcal{T}_{\bar{e}})\, ds = \sum_{\ell \in I^+ }\kappa(\ell) \left(j^+_{\sigma^{k_i}(x)}\left(\sigma_*^{(k_i)}i^+_\Gamma(f)\right)\right)_\ell + \sum_{\ell \in I^0 }\kappa(\ell) \left(j^+_{\sigma^{k_i}(x)}\left(\sigma_*^{(k_i)}i^+_\Gamma(f)\right)\right)_\ell +O(e^{2\varepsilon' k_i}),$$
which, after rearranging, using the triangle inequality, and rearranging again, we get
\begin{equation}
  \label{eqn:liminf1}
  \begin{split}
    \left|\int_{\mathcal{P}_\varepsilon^i(\bar{e})} f\circ \varphi_s(\mathcal{T}_{\bar{e}})\, ds\right| &\geq \left|\sum_{\ell \in I^+ }\kappa(\ell) \left(j^+_{\sigma^{k_i}(x)}\left(\sigma_*^{(k_i)}i^+_\Gamma(f)\right)\right)_\ell\right|\\
    &\hspace{.25in}- \left|\sum_{\ell \in I^0 }\kappa(\ell) \left(j^+_{\sigma^{k_i}(x)}\left(\sigma_*^{(k_i)}i^+_\Gamma(f)\right)\right)_\ell +O(e^{2\varepsilon' k_i})\right| \\
    &\geq  C^+\|\sigma^{k_i}_*\pi^+_{r,x}i^+_\Gamma f\|
  \end{split}
\end{equation}
for all $i$ and some $C^+>0$ small enough. Recalling that $T_i$ is proportional to $\theta_{(k_i)_x}^{-1}$ and using (\ref{eqn:asympContr}):
\begin{equation}
  \label{eqn:liminf2}
  \begin{split}
    \limsup_{i\rightarrow \infty}\frac{\log \left|\displaystyle\int_{\mathcal{P}_\varepsilon^i(\bar{e})} f\circ \varphi_s(\mathcal{T}_{\bar{e}})\, ds\right|}{\log T_i} &\geq  \limsup_{i\rightarrow \infty} \frac{\log C^+\|\sigma^{k_i}_*\pi^+_{r,x}i^+_\Gamma f\|}{\log T_i} \\
    &= \limsup_{i\rightarrow \infty} \frac{k_i}{\log T_i}\frac{\log \left\|\sigma^{k_i}_*\pi^+_{r,x}i^+_\Gamma f\right\|}{k_i} = \frac{d}{\lambda_1^+}\lambda_r^+.
  \end{split}
\end{equation}
\section{Solenoids and the Denjoy-Koksma inequality}
\label{sec:DK}
For a function $f:X\rightarrow\mathbb{R}$ on a Cantor set $X$ and a clopen subset $C\subset X$, define
$$\mathrm{Var}(f,C) := \sup_{x,y\in C}|f(x)-f(y)|$$
and, for a partition $P = \{P_1,\dots ,P_n\}$ of $X$ into disjoint clopen subsets, define
$$\mathrm{Var}(f,P) = \sum_{i=1}^n \mathrm{Var}(f,P_i).$$
A Cantor set naturally carries a metric structure. In fact, Cantor sets carry \emph{ultrametric} structures, and so any ball $B_\epsilon(x)\subset X$ is a clopen set. Let $d_X:X^2\rightarrow \mathbb{R}$ be an ultrametric on $X$. For any set $C\subset X$ let $\mathrm{diam}\, C =\sup\{d_X(x,y):x,y\in C\} $. Let $\mathcal{P}_\epsilon$ be the set of all partitions $\{P_1,\dots, P_k\}$ of $X$ by clopen sets with $\mathrm{diam}\, P_i\leq \epsilon$ for all $i$.
Finally, let
$$\mathrm{Var}_\epsilon(f) = \sup_{P\in\mathcal{P}_\epsilon} \mathrm{Var}(f,P)\hspace{.75in}\mbox{ and } \hspace{.75in}\mathrm{Var}(f) = \sup_{\epsilon>0}\mathrm{Var}_\epsilon(f).$$
A function $f:X\rightarrow \mathbb{R}$ on a Cantor set $X$ has \textbf{bounded variation} if $\mathrm{Var}(f)<\infty$. Note that if $f$ is a locally constant function on a Cantor set, then $\mathrm{Var}\, (f)=0$, so it is of bounded variation.
\begin{definition}
Let $\Omega$ be a the tiling space of an aperiodic, repetitive tiling of finite local complexity. A continuous function $f:\Omega\rightarrow \mathbb{R}$ has \textbf{bounded variation} if there is a $V_f< \infty$ such that $ \mathrm{Var}(f|_{\mho})\leq V_f$ for all transversals $\mho\subset \Omega$ which are Cantor sets.
\end{definition}
The set of continuous functions on $\Omega$ with bounded variation is denoted by $\mathrm{BV}(\Omega)$. Note that if $f$ is a transversally locally constant function then it is in $\mathrm{BV}(\Omega)$. Let $\bar{q} = (q_1,q_2,\dots)\in \mathbb{N}^{\mathbb{N}}$, where $q_k>1$ for all $k$. For any such $\bar{q}$ we will denote $q_{(n)} = q_1\cdots q_{n}$.
\begin{definition}
A \textbf{$d$-dimensional solenoid} is the tiling space $\Omega_{\bar{q}}$ associated to a family of substitutions $\mathcal{F}$ on a single prototile $t_1 = \left[-\frac{1}{2},\frac{1}{2}\right]^d$. The Bratteli diagram $\mathcal{B}_{\bar{q}}$ for such tiling spaces have a single vertex at every level and $|\mathcal{E}_k| = q_k^d$ for all $k\in\mathbb{Z}$, and it is also required here that $\theta_{e} = q^{-d}_k$ for any $e\in\mathcal{E}_k$. In this case the family $\mathcal{F}$ is allowed to be infinite.
\end{definition}
\begin{remark}
  The definition for a solenoid above is slightly more general than the usual definition of a solenoid as an inverse limit of $\mathbb{T}^d$ under maps of the form $q_n\cdot \mbox{Id}$.
  \end{remark}

The goal of this section is to prove a type of bound known as a \textbf{Denjoy-Koksma inequality} (\cite[\S VI.3]{herman:cercle}) for solenoids.
\begin{theorem}[Denjoy-Koksma inequality for solenoids]
  Let $\Omega_{\bar{q}}$ be a $d$-dimensional solenoid. Then for any $f\in \mathrm{BV}(\Omega_{\bar{q}})$ and $p\in \Omega_{\bar{q}}$,
  $$\left| \int_{[0,q_n]^d} f\circ \varphi_s(p)\, ds - q_{(n)}^d\int_{\Omega_{\bar{q}}} f\, d\mu \right|\leq \mathrm{Var}(f)$$
  for all $n>0$.
\end{theorem}
\begin{remark}
  It seems reasonable to conjecture that a Denjoy-Koksma inequality holds for any tiling space $\Omega_x$ obtained from compatible and uniformly expanding substitutions with $AF(\mathcal{B}^+_x)$ is UHF (see \cite[\S III.5]{davidson:book}). It seems like for $d=1$ the proof below can be combined with the usual intertwining arguments to give a proof.
\end{remark}
\begin{proof}
  Let $X^+_{\bar{q}}:=X^+_{\mathcal{B}_{\bar{q}}}$. Since any substitution in the family $\mathcal{F}$ forces the border, the map $\Delta_{\bar{q}}:X^+_{\bar{q}}\rightarrow \mho_{\bar{q}}$ is a homeomorphism of Cantor sets. As such, the topology of $\mho_{\bar{q}}$ is generated by the image of cylinder subsets of $X^+_{\bar{q}}$ under the map $\Delta_{\bar{q}}$, and the ultrametric structure of $\mho_{\bar{q}}$ is inherited from that of $X_{\bar{q}}$. As such, for every $k>0$, there are $q_{(k)}^{d}$ pairwise-disjoint cylinder sets $\mathcal{C}^k_i\subset \mho_{\bar{q}}$, parametrized by $ i\in\{1,\dots, q_{(k)}\}^d$, one for each path $p\in \mathcal{E}_{0,k}$, whose union is $\mho_{\bar{q}}$. Moreover, since it is well known that $X_{\mathcal{B}_{\bar{q}}}^+$ admits a unique tail-invariant Borel probability measure, by Proposition \ref{prop:measBij}, we have that $\mathrm{diam}\,\mathcal{C}^k_i =  \nu(\mathcal{C}^k_i) = q^{-d}_{(k)}$ for all $i$ for the unique holonomy-invariant measure $\nu$ on $\mho_{\bar{q}}$.

  For any $\bar{e} = (\bar{e}^-,\bar{e}^+)\in X_{\bar{q}}$, the $k^{th}$ approximant $\mathcal{P}_k(\bar{e}^+)$ is a tiled cube of side length $q_{(k)}$ containing the origin, and it is tiled by $q_{(k)}^d$ tiles isometric to $[0,1]^d$. For $\mathcal{T}_{\bar{e}} = \Delta_{\bar{q}}(\bar{e})\in \Omega_{\bar{q}}$ there exists a vector $\tau_{\bar{e}}\in \left[-\frac{1}{2},\frac{1}{2}\right]^d$ such that $\varphi_{\tau_{\bar{e}}}(\mathcal{T}_{\bar{e}})\in \mho_{\bar{q}}$. By Lemma \ref{lemma:deform}, there exist $q_{(k)}^d$ vectors $\tau_1,\dots, \tau_{q_{(k)}^d}$ such that $\varphi_{\tau_{i}}(\mathcal{T}_{\bar{e}})\in \mathcal{C}^k_i$. In other words, the points $\{\varphi_{\tau_{i}}(\mathcal{T}_{\bar{e}})\}_i$ $q_{(k)}^{-d}$-equidistribute in $\mho_{\bar{q}}$.

  In fact, more is true: the vectors $\tau_1,\dots, \tau_{q_{(k)}^d}$ can be chosen to be nice elements of $\mathbb{Z}^d$. In particular, one can choose them to be the elements of the set $\{0,\dots, q_{(k)}-1\}^d$. First, note that for any $s\in \{0,\dots, q_{(k)}-1\}^d$, $\varphi_{s+\tau_{\bar{e}}} (\mathcal{T}_{\bar{e}})\in \mho_{\bar{q}}$. This follows from the fact that there is a single prototile (a unit cube) in the tiling and its center is the puncture. Thus, since $\varphi_{\tau_{\bar{e}}}(\mathcal{T}_{\bar{e}})\in \mho_{\bar{q}}$, it follows that $\varphi_{s+\tau_{\bar{e}}} (\mathcal{T}_{\bar{e}})\in \mho_{\bar{q}}$. Moreover, for any $s\neq s'\in \{1,\dots, q_{(k)}\}^d$, it follows that $\Delta_{\bar{q}}^{-1}(\varphi_{s+\tau_{\bar{e}}} (\mathcal{T}_{\bar{e}}))|_k \neq \Delta_{\bar{q}}^{-1}(\varphi_{s'+\tau_{\bar{e}}} (\mathcal{T}_{\bar{e}}))|_k$, so they $q_{(k)}^{-d}$-equidistribute in $\mho_{\bar{q}}$.
 
  Now recall the proof of the classical Denjoy-Koksma inequality now for the dynamics restricted to $\mho_{\bar{q}}$ \cite[Th\'eor\`eme VI.3.1]{herman:cercle}. For $\bar{e}\in\mho_{\bar{q}}$:
  \begin{equation}
    \label{eqn:DKtransv}
    \begin{split}
   &   \left| \sum_{i\in \left\{0,\dots, q_{(k)}-1\right\}^d} f\circ\varphi_i(\mathcal{T}_{\bar{e}}) -q_{(k)}^{d}\int_{\mho_{\bar{q}}}f\, d\nu \right|= \left|\sum_{i\in \left\{0,\dots, q_{(k)}-1\right\}^d}  \frac{1}{\nu(\mathcal{C}^k_i)}\int_{\mathcal{C}^k_i}\left(f\circ \varphi_i(\mathcal{T}_{\bar{e}}) - f(x)\right)\, d\nu (x) \right| \\
      &\hspace{.4in}\leq \sum_{i\in \left\{0,\dots, q_{(k)}-1\right\}^d}  \frac{1}{\nu(\mathcal{C}^k_i)}\int_{\mathcal{C}^k_i}\left|f\circ \varphi_i(\mathcal{T}_{\bar{e}}) - f(x)\right|\, d\nu (x) \leq \sum_{i\in \left\{0,\dots, q_{(k)}-1\right\}^d}\sup_{y,z\in \mathcal{C}^k_i}|f(y) - f(z)| \\
      &\hspace{.4in}\leq \sum_{i\in \left\{0,\dots, q_{(k)}-1\right\}^d}\mathrm{Var}(f,\mathcal{C}^k_i)\leq \mathrm{Var}_{q^{-d}_{(k)}}(f).
    \end{split}
  \end{equation}

  Up to this point everything has been done with reference to the transversal at zero $\mho_{\bar{q}}$. It turns out that for every point $y\in [0,1)^d$ there is an associated transversal $\mho^y_{\bar{q}}\subset \Omega_{\bar{q}}$ obtained by translating $\mho_{\bar{q}} = \mho^0_{\bar{q}}$ by $y$. By composition with this translation the map $\Delta_{\bar{q}}$ is a homeomorphism between $X^+_{\bar{q}}$ and $\mho_{\bar{q}}^y$, and so for any $k>0$ there is a partition $\{\mathcal{C}^k_i(y)\}_i$ of $\mho_{\bar{q}}^y$ by $q^d_{(k)}$ cylinder sets of measure $\nu_y(\mathcal{C}^k_i(y)) = q^{-d}_{(k)}$, where the measure $\nu_y$ on $\mho^y_{\bar{q}}$ is the translate of the measure $\nu$ on $\mho_{\bar{q}}$. Thus the same arguments leading to (\ref{eqn:DKtransv}) hold for the transversal $\mho^y_{\bar{q}}$ and so it follows that for $\mathcal{T}_{\bar{e}}\in \mho^y_{\bar{q}}$
    \begin{equation}
      \label{eqn:DKtransv2}
      \left| \sum_{i\in \left\{0,\dots, q_{(k)}-1\right\}^d}f\circ\varphi_i(\mathcal{T}_{\bar{e}}) -q_{(k)}^{d}\int_{\mho_{\bar{q}}^y}f\, d\nu_y \right|\leq \mathrm{Var}_{q_{(k)}^{-d}}(f).
    \end{equation}
Finally, note that if $p\in \Omega_{\bar{q}}$ and $y\in [0,1)^d$ such that $p\in \mho^y_{\bar{q}}$ then
    $$\int_{[0,q_k]^d} f\circ \varphi_s(p)\, ds = \int_{\left[0,1\right]^d}   \sum_{i\in \left\{0,\dots, q_{(k)}-1\right\}^d}f\circ\varphi_{i+s}(p)\, ds.$$
Putting it all together, let $p\in \Omega_{\bar{q}}$ and $y\in [0,1)^d$ such that $p\in \mho^y_{\bar{q}}$. Then:
    \begin{equation}
      \begin{split}
&        \left| \int_{[0,q_k]^d} f\circ \varphi_s(p)\, ds - q_k^d\int_{\Omega_{\bar{q}}} f\, d\mu\right| \\
        &\hspace{1in}= \left| \int_{\left[0,1\right]^d}   \sum_{i\in \left\{0,\dots, q_{(k)}-1\right\}^d}f\circ\varphi_{i+s}(p)\, ds -q_{(k)}^{d}\int_{\left[0,1\right]^d}\int_{\mho_{\bar{q}}^{\varphi_s(y)}}f\, d\nu_{\varphi_s(y)}\, ds \right| \\
        &\hspace{1in}\leq  \int_{\left[0,1\right]^d} \left|   \sum_{i\in \left\{0,\dots, q_{(k)}-1\right\}^d}f\circ\varphi_{i+s}(p)\, -q_{(k)}^{d}\int_{\mho_{\bar{q}}^{\varphi_s(y)}}f\, d\nu_{\varphi_s(y)} \right|ds  \\
        &\hspace{1in}\leq \int_{\left[0,1\right]^d}\mathrm{Var}_{q_{(k)}^{-d}}(f) \, ds = \mathrm{Var}_{q_{(k)}^{-d}}(f)\leq \mathrm{Var}(f).
      \end{split}      
    \end{equation}
\end{proof}
\section{Random Schr\"odinger operators}
\label{sec:schrod}
This section will focus on applications of the results of \S \ref{sec:erg} to algebras of operators coming from the tiling spaces obtained by collections of substitution rules. Although it is natural in such cases to focus on the $C^*$-algebras of operators obtained, here the focus is on $*$-algebras which are dense in the $C^*$-algebras of usual interest. This is because the traces obtained are only densely defined and one loses all but one trace by going to the completion $C^*$-algebras. This is mentioned for the curious reader wondering how one completes the algebras constructed; it is not relevant for the work here. However, the reader can see, for example, \cite{bellissard:K} for how the use of operator algebras enters the study of aperiodic media from a mathematical physics point of view; see also \cite{KP:survey, LS:algebras} for several uses of $C^*$-algebras in the study of tilings. The $*$-algebras used here will be dense subalgebras of the ones used in \cite{LS:algebras}.

For a family $\mathcal{F}$ of uniformly expanding and compatible substitutions defined on the same set of prototiles and $x\in X_\mathcal{F}$ let $\mathcal{B}_x$ be the associated Bratteli diagram as constructed in \S \ref{sec:brat} and assume $\mathcal{B}_x$ is minimal. Recall that by construction, any tile $t$ on any tiling $\mathcal{T}\in\Omega$ has a distinguished point in its interior, and they correspond to the placement of the origin inside of the prototiles $\{t_1,\dots, t_M\}$. These distinguished points are called \textbf{punctures} in \cite{kellendonk:gap}. Once the puntures have been chosen in the interior of the prototiles, there exists a $\varrho>0$ such that any ball of radius less than $\varrho$ centered at the puncture of a tile $t\in\mathcal{T}\in\Omega_x$ does not intersect the boundary of $t$, and this holds for all $x\in X_\mathcal{F}$ and $\mathcal{T}\in \Omega_x$. Let $\Lambda_\mathcal{T}$ be the set of punctures of $\mathcal{T}$, that is, the union of all distinguished points of all tiles of $\mathcal{T}$ and define
      $$ \mathcal{G}_{x} := \left\{(p,\mathcal{T}',q)\in \mathbb{R}^d\times \Omega_x\times\mathbb{R}^d: p,q\in\Lambda_{\mathcal{T}'}\right\}.$$
      \begin{definition}
        A kernel of finite range is a function $k\in C(\mathcal{G}_x)$ such that
        \begin{enumerate}
        \item $k$ is bounded;
        \item $k$ has finite range. In other words there is a $R_k>0$ such that $k(p,\mathcal{T},q) = 0$ whenever $|p-q|>R_k$;
        \item $k$ is $\mathbb{R}^d$-invariant:  $k(p-t,\varphi_t(\mathcal{T}),q-t) = k(p,\mathcal{T},q)$ for any $t\in\mathbb{R}^d$.
        \end{enumerate}
      \end{definition}
      The set of all kernels of finite range associated to $\Omega_x$ are denoted by $\mathcal{K}_x^{fin}$. For any $k\in\mathcal{K}_x^{fin}$ there is a family of representations $\{\pi_\mathcal{T}\}_{\mathcal{T}\in\Omega_x}$ in $\mathcal{B}(\ell^2(\Lambda_\mathcal{T}))$ defined, for $k\in\mathcal{K}_x^{fin}$ by
      $$\langle K_\mathcal{T} \delta_p,\delta_q\rangle = \langle(\pi_\mathcal{T}k)\delta_p,\delta_q\rangle = k(p,\mathcal{T},q).$$
      The family $\{K_\mathcal{T}\}$ parametrized by $\Omega_x$ is bounded in the product $\prod_{\mathcal{T}\in\Omega_x}\mathcal{B}(\ell^2(\Lambda_\mathcal{T}))$. Defining a convolution product as
      $$(a\cdot b)(p,\mathcal{T},q) = \sum_{x\in \Lambda_\mathcal{T}}a(p,\mathcal{T},x)b(x,\mathcal{T},q)$$
      and involution by $k^*(p,\mathcal{T},q)=\overline{k(q,\mathcal{T},p)}$, $\mathcal{K}_x^{fin}$ has the structure of a $*$-algebra. It follows that the map $\pi:\mathcal{K}_x^{fin}\rightarrow \prod_{\mathcal{T}}\mathcal{B}(\ell^2(\Lambda_\mathcal{T}))$ is a faithful $*$-representation. The image is denoted by $\mathcal{A}_x^{fin}$ and it is the algebra of operators of finite range. The completion of this algebra is denoted by $\mathcal{A}_x$.

      \begin{definition}
        The set of \textbf{Lipschitz kernels} of finite range consists of kernels $k\in\mathcal{K}_x^{fin}$ for which there are constants $R_k, L_k>0$ such that if for two $\mathcal{T}_1,\mathcal{T}_2\in\Omega_x$ one has that $B_{R_k}(0)\cap \Lambda_{\mathcal{T}_1} = B_{R_k}(0)\cap \Lambda_{\mathcal{T}_2}$ then for any $p,q\in B_{R_k}(0)\cap \Lambda_{\mathcal{T}_1}$ one has that $|k(p,\mathcal{T}_1,q)-k(p,\mathcal{T}_2,q)|\leq L_k d(\mathcal{T}_1,\mathcal{T}_2)$.
      \end{definition}
      The kernels in the above definition carry the label \emph{Lipschitz} since they will be connected to Lipschitz functions on the tiling space $\Omega_x$; see Lemma \ref{lem:lipMap} below.
      
      The set of Lipschitz kernels of finite range is denoted by $\mathcal{LK}_x^{fin}\subset \mathcal{K}_x^{fin}$. The image of $\mathcal{LK}_x^{fin}$ is denoted by $\mathcal{LA}_x^{fin} = \pi\mathcal{LK}_x^{fin}\subset \mathcal{A}_x^{fin}$ and it is the set of \textbf{Lipschitz operators of finite range}. It should be pointed out that most operators of interest in mathematical physics, such as operators of the form\footnote{A simple example to consider in one dimension is as follows. Let $\mathcal{T}$ be a tiling of $\mathbb{R}$ by $N$ different tile types and $\Lambda_\mathcal{T}$ the collection of endpoints of tiles of $\mathcal{T}$. There is an obvious, order-preserving labeling of $\Lambda$ by $\mathbb{Z}$. For $i\in\{1,\dots, N\}$ and $\lambda\neq 0$, consider the operator $H_{i,\lambda} = \triangle + \lambda V_i$, where $\triangle$ is the discrete Laplacian on $\mathbb{Z}$ and the localized potential $V_i$ is defined by $V_i(p) =  p$ if $p\in \Lambda$ is a left endpoint of a tile of type $i$, and otherwise $V_i(p)=0$. Similar constructions can be made for tilings in higher dimensions by considering the graph $\mathcal{G}_\mathcal{T}$ given by the Delaunay triangulation of $\Lambda_\mathcal{T}$, considering the graph Laplacian $\triangle_{\mathcal{G}_\mathcal{T}}$ on $\mathcal{G}_\mathcal{T}$ and considering an operator of the form $H = \triangle_{\mathcal{G}_\mathcal{T}} + V$, where $V$ is a localized potential depending only on the local pattern aroung $p\in\Lambda_\mathcal{T}$.} $H = \triangle + V$, where $V$ is a ``localized'' potential on defined on $\mathcal{T}$, are contained in the set $\mathcal{LA}_x^{fin}$.

      Let $u:\mathbb{R}^d\rightarrow \mathbb{R}$ be a smooth non-negative (bump) function of integral 1, compactly supported in a disk of radius less than $\varrho$. This defines a family of functions $w_{u,\mathcal{T}}:\mathcal{LA}_{x}^{fin}\rightarrow C^\infty(\mathbb{R}^d)$ parametrized by $\Omega_x$ as follows. For $A = \pi k \in \mathcal{A}_x^{fin}$ and $A_\mathcal{T} = \pi_\mathcal{T}k\in\mathcal{B}(\ell^2(\Lambda_\mathcal{T}))$, let $f_{A_\mathcal{T}}^u$ be defined by
      $$f_{A_\mathcal{T}}^u(t) = w_{u,\mathcal{T}}(A)(t) = \sum_{p\in\Lambda_\mathcal{T}}A_{\mathcal{T}}(p,p) u(p-t).$$
      \begin{lemma}
        \label{lem:lipMap}
        For any $A\in \mathcal{L}\mathcal{A}_x^{fin}$ and $\mathcal{T}\in\Omega_x$ there exists a Lipschitz function $h = h_A^u\in\mathcal{L}(\Omega_x)$ such that $f_{A_\mathcal{T}}^u(t) = h\circ\varphi_t(\mathcal{T})$.
      \end{lemma}
      \begin{proof}
The assignment $\mathcal{T}\mapsto f^u_{A_\mathcal{T}}(0)$ defines a function $f^u_{A_\cdot}(0): \mho_x\rightarrow\mathbb{R}$. For $\mathcal{T},\mathcal{T}'\in\mho_x$ one then has for any $k\in \mathcal{LK}_x^{fin}$
        $$\left|f^u_{A_\mathcal{T}}(0) - f^u_{A_{\mathcal{T}'}}(0)\right| = \left| k(0,\mathcal{T},0) - k(0,\mathcal{T}',0) \right|u(0)\leq u(0)L_k d(\mathcal{T}_1,\mathcal{T}_2),$$
so this is a Lipschitz function on $\mho_x$ with Lipschitz constant $L_ku(0)$.

The function $\mathcal{T}\mapsto f^u_{A_\cdot}(0)$ can be extended to $\Omega_x$ by choosing a neighborhood $U$ of $\mho_x$ of size $r_u$ and a product chart $\phi_u:U\rightarrow B_{r_u}(0)\times \mho_x$ and noting that the function defined by $h = \phi_u^*\bar{u}$, where $\bar{u}(t,\mathcal{T}) = f^u_{A_\mathcal{T}}(0)u(t)$ with $\|t\|<r_u$, defines a Lipschitz function on $\Omega_x$. That this gives $f^u_{A_\mathcal{T}}(t) = h\circ \varphi_t(\mathcal{T})$ follows from the $\mathbb{R}^d$ invariance of the kernel $k$ used to define $A$.
      \end{proof}
      Let $\mathcal{M}_u:\mathcal{LA}_x^{fin}\rightarrow \mathcal{L}(\Omega_x)$ be the map given by Lemma \ref{lem:lipMap} and denote the composition $\Upsilon_{u,\Gamma} := j^+_x\circ i^+_\Gamma \circ \mathcal{M}_u:\mathcal{LA}_x^{fin}\rightarrow LF(\mathcal{B}^+_x)$. We can define functionals $\tau'_i:\mathcal{LA}_x^{fin}\rightarrow \mathbb{C}$ by pullback $\tau_i' = \Upsilon_{u,\Gamma}^*\tau_i^+$, i.e., $\tau_i'(A) = \tau_i^+(\Upsilon_{u,\Gamma}(A))$, for $A\in \mathcal{LA}_x^{fin}$, where $\tau_i^+\in T^+_i(x)$. The functionals $\tau'_i$ may or may not be traces. By \cite[Proposition 1]{ST:traces}, we know some cases when they are.
      \begin{proposition}
        \label{prop:pullTrace}
         Let $\mathcal{F}_1,\dots, \mathcal{F}_N$ be a collection of uniformly expanding and compatible substitution rules on a set of prototiles $t_1,\dots, t_M$ and $\mu$ a minimal, $\sigma$-invariant ergodic Borel probability measure on $X_\mathcal{F}$. Then for $\mu$-almost every $x$, for a spanning system of patches $\Gamma$ on $\Omega_x$, the functional $\tau'_i = \Upsilon_{u,\Gamma}^*\tau_i^+$ is a trace if $\frac{\lambda_i^+}{\lambda_1^+}>\frac{d-1}{d}$. So $\Upsilon_{u,\Gamma}$ induces a map on traces
         $$ \Upsilon_{u,\Gamma}^* :\tr (\mathcal{B}^+_x)^{++}\rightarrow \tr(\mathcal{LA}_x^{fin}),$$
         where $\tr (\mathcal{B}^+_x)^{++}$ is the subspace of $\tr (\mathcal{B}^+_x)$ generated by traces $\tau_i^+$ which satisfy $\frac{\lambda_i^+}{\lambda_1^+}>\frac{d-1}{d}$.
       \end{proposition}
       \subsection{Proof of Theorem \ref{thm:opDevs}}
       Let $d^{++}_\mu$ be the dimension of the subspace $\tr (\mathcal{B}^+_x)^{++}$. Define the $d^{++}_\mu$ traces in $\tr(\mathcal{LA}_x^{fin})$ to be $\{\tau_1,\dots , \tau_{d^{++}_\mu}\}$, where $\tau_i\in \Upsilon_{u,\Gamma}^*T^+_i(x)$ is any non-zero element. Now pick $A_\mathcal{T}\in \mathcal{LA}_x^{fin}$, a good Lipschitz domain $B$ and $T>0$.
       First, note that for two smooth bump functions $u,u'$ of compact support in a ball of radius less than $\rho$ and integral 1, it follows that
       \begin{equation}
         \label{eqn:bumpInd}
         \begin{split}
           &\left|\int_{\mathcal{O}^-_\mathcal{T}(E)}\mathcal{M}_u A\circ \varphi_t(\mathcal{T})- \mathcal{M}_{u'} A\circ \varphi_t(\mathcal{T}) \, dt \right| = 0,\hspace{.4in}\mbox{ and }\\
& \hspace{1.5in}          \left|\int_{E}\mathcal{M}_u A\circ \varphi_t(\mathcal{T})- \mathcal{M}_{u'} A\circ \varphi_t(\mathcal{T}) \, dt \right| = O(|\partial E|).
           \end{split}
       \end{equation}
       for any measurable $E$ of finite volume.
       In addition, it follows that
       \begin{equation}
         \label{eqn:traceId}
         \begin{split}
         \left|\mathrm{tr}\left(A_\mathcal{T}|_{\mathcal{O}^-_\mathcal{T}(E)}\right) - \int_{\mathcal{O}^-_\mathcal{T}(E)}  \mathcal{M}_u A\circ \varphi_t(\mathcal{T})\, dt\right| &= 0\\
         \left|\mathrm{tr}(A_\mathcal{T}|_{E}) - \int_{E}  \mathcal{M}_u A\circ \varphi_t(\mathcal{T})\, dt\right| &= O(|\partial E|).
         \end{split}
       \end{equation}
for any measurable set $E$ of finite volume,  where the second estimate is from \cite[Equation (22)]{ST:traces}. Thus, if $\tau_i(A) = 0$ for all $i=1,\dots, r$ for some $r< d^{++}_\mu$ but $\tau_r(A)\neq 0$, by (\ref{eqn:ergLong}) it follows that for any $\varepsilon>0$
\begin{equation}
\begin{split}
   |\mathrm{tr}(A_\mathcal{T}|_{T\cdot B)})| &= \left|\int_{T\cdot B}  \mathcal{M}_u A\circ \varphi_t(\mathcal{T})\, dt\right|+ O(\partial(T\cdot B)) \\
   &\hspace{.3in}\leq \max\left\{C_{\varepsilon,A} \tau_r(A)T^{d\frac{\lambda_r^+}{\lambda_1^+}+d\varepsilon},O(T^{d-1})\right\}
  \end{split}
  \end{equation}
independent of which bump function $u$ was used by (\ref{eqn:bumpInd}). Thus by (\ref{eqn:ergInt4}) and (\ref{eqn:ergInt5}), if $d\lambda_r^+\geq (d-1)\lambda_1^+$, then
$$\limsup_{T\rightarrow \infty}\frac{\log |\mathrm{tr}(A_\mathcal{T}|_{T\cdot B})|}{\log T} \leq \frac{\lambda^+_r}{\lambda_1^+}d.$$
For $\varepsilon'>0$ we choose a set $B_{\varepsilon'}$ as in \S \ref{subsec:lowBnds} along with the sequence of times $T_i\rightarrow \infty$ and vectors $\tau_i\in\mathbb{R}^d$. By construction, $T_i\cdot( B_{\varepsilon'}+\tau_i) = \mathcal{O}^-_\mathcal{T}(E_i)$, where $E_i\subset \mathbb{R}^d$ is some measurable subset of finite volume. Thus, the results of \S \ref{subsec:lowBnds} along with (\ref{eqn:bumpInd})-(\ref{eqn:traceId}) imply that 
$$\limsup_{i\rightarrow \infty}\frac{\log |\mathrm{tr}(A_\mathcal{T}|_{T_i\cdot ( B_{\varepsilon'}+\tau_i)})|}{\log T_i} = \frac{\lambda^+_r}{\lambda_1^+}d.$$
\section{Variations on half hexagons}
\label{sec:experiments}
Let me close by giving some experimental results. Consider the two substitution rules on the half hexagons in Figure \ref{fig:HalfHexes} in the introduction and which were studied in \S \ref{sec:interlude}. The first substitution rule depicted is the classical substitution rule in the half-hexagon with expansion constant 2. The eigenvalues of the corresponding substitution matrix are $4,2,1,1,-1,-1$. The second substitution rules has expansion constant 4 and the eigenvalues for the corresponding substitution matrix are $16, 7\pm i\sqrt{3}, 2,2,2$. Note that $|7\pm i\sqrt{3}|> 4 = \sqrt{16}$, so the second substitution rule has ``rapidly-expanding'' eigenvalues.

For $p\in (0,1)$, let $\mu_p$ be the Bernoulli measure on $\Sigma_2$ which gives the cylinder set $\mu_p(C_1) = p$ and $\mu_p(C_2) = 1-p$, where $C_i = \{x\in \Sigma_2: x_1 = i\}$. The typical points for the measure $\mu_p$ then give tiling spaces $\Omega_x$ which are obtained from tilings which were constructed, on average by applications of the first substitution in Figure \ref{fig:HalfHexesGIFS} with probability $p$ and the second substitution from Figure \ref{fig:HalfHexesGIFS} with probability $1-p$. Note that from the graphs in Figure \ref{fig:HalfHexesGIFS} it is easy to recover the two matrices which are used to compute the trace cocycle.

Figure \ref{fig:spectrum} shows the (normalized) spectrum as a function of $p$. It is normalized because what is plotted are the ratios $2\lambda^+_i/\lambda^+_1$, which are the relevant exponents in the main results of this paper. Perhaps not surprisingly, when $p>1/2$, there seem to be a pair of (normalized) Lyapunov exponents greater than 1, meaning that there are non-trivial deviations of ergodic averages for tilings in a typical tiling space $\Omega_x$ with respect to the measure $\mu_p$. In particular, as pointed out in the first item of Remark \ref{rem:2}, this shows the rate of convergence in the Shubin-Bellissard formula for the integrated density of states for any Lipschitz kernels of finite range.
\begin{figure}[h]
 \centering
 \includegraphics[width = 5.5in]{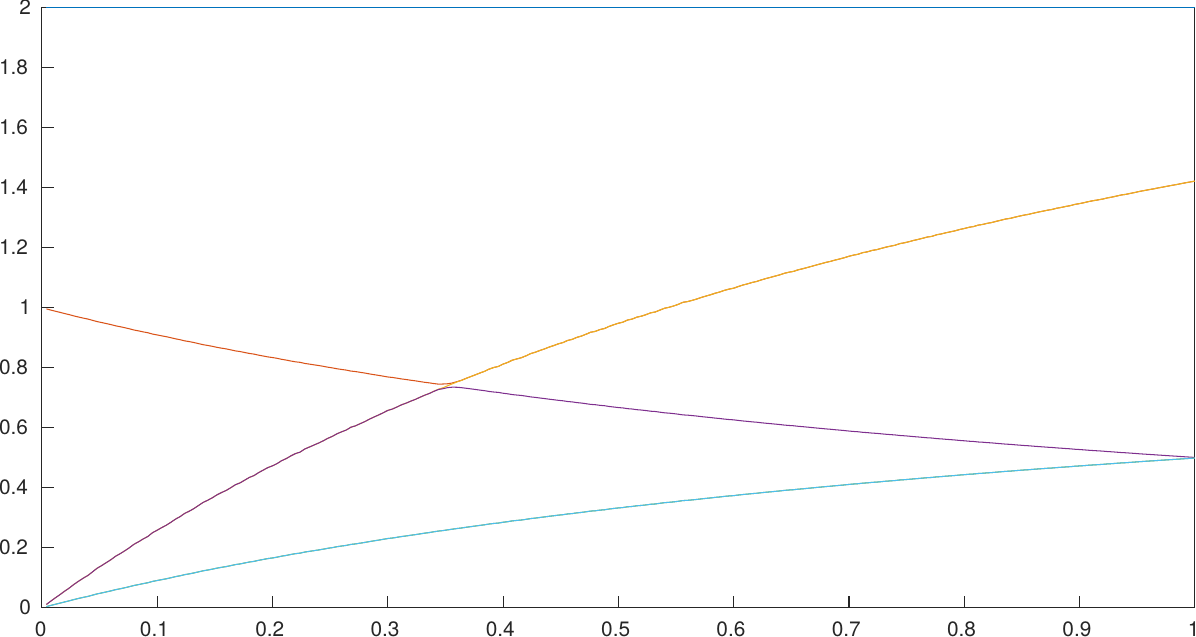}
  \caption{ Lyapunov spectrum for the measures $\mu_p$ as a function of $p$.}
  \label{fig:spectrum}
\end{figure}

\bibliographystyle{amsalpha}
\bibliography{biblio}

\end{document}